\documentclass{article} 
\usepackage[utf8]{inputenc} 
\usepackage{geometry} 
\usepackage{graphicx} 
\usepackage{booktabs} 
\usepackage{array} 
\usepackage{paralist} 
\usepackage{verbatim} 
\usepackage{subfig} 
\usepackage{amsmath, amssymb, layout, amsthm, appendix}
\usepackage{fancyhdr} 
\usepackage{hyperref}
\usepackage{xfrac, faktor}
\usepackage{comment}
\usepackage[all,cmtip]{xy}
\newtheorem{theorem}{Theorem}[subsection]
\newtheorem{definition}[theorem]{Definition}
\newtheorem{lemma}[theorem]{Lemma}
\newtheorem{proposition}[theorem]{Proposition}

\theoremstyle{remark}
\newtheorem{remark}[theorem]{Remark}
\theoremstyle{definition}
\newtheorem{example}[theorem]{Example}
\numberwithin{equation}{subsection}

\usepackage{sectsty}
\allsectionsfont{\sffamily\mdseries\upshape} 

\usepackage[nottoc,notlof,notlot]{tocbibind} 
\usepackage[titles,subfigure]{tocloft} 

\title{Genus zero Gromov-Witten axioms via Kuranishi atlases}
\author{Robert Castellano}
\date{}

\begin{document}
\maketitle
\abstract{A Kuranishi atlas is a structure used to build a virtual fundamental class on moduli spaces of $J$-holomorphic curves. They were introduced by McDuff and Wehrheim to resolve some of the challenges in this field. This paper completes the construction of genus zero Gromov-Witten invariants using Kuranishi atlases and proves the Gromov-Witten axioms of Kontsevich and Manin. To do so, we introduce the notion of a transverse subatlas, a useful tool for working with Kuranishi atlases.}
\tableofcontents

\section{Introduction}
The theory of Kuranishi atlases was developed in \cite{mwtop, mwfund, mwiso} by McDuff and Wehrheim. In these papers, they develop the abstract framework for the theory of Kuranishi atlases and construct a virtual fundamental class for a space $X$ admitting a \textit{smooth} Kuranishi atlas (see Theorem \ref{theoremb}). The prototype of the space $X$ for which this theory was developed is a compactified moduli space of $J$-holomorphic curves. The technique used is that of finite dimensional reduction and is based on the work by Fukaya-Ono \cite{fo} and Fukaya-Oh-Ohta-Ono \cite{fooo}.  In \cite{notes} McDuff describes the construction of a Kuranishi atlas for the Gromov-Witten moduli space of closed genus zero curves. However, using the ``weak" gluing theorem of \cite{JHOL}, one does not obtain a smooth Kuranishi atlas, but rather a \textit{stratified smooth} Kuranishi atlas. In \cite{notes} McDuff describes how to adapt the theory of Kuranishi atlases to stratified smooth atlases and construct a virtual fundamental class in the case the space has virtual dimension zero. In this paper we extend these results to arbitrary virtual dimension and prove the axioms for Gromov-Witten invariants formulated by Kontsevich-Manin \cite{axioms}.

\subsection{Statement of main results}
In \cite{gluingme}, the author proves a stronger ``$C^1$ gluing theorem" and uses this to show that the genus zero Gromov-Witten moduli space admits a stronger structure than a stratified smooth structure called a \textit{$C^1$ stratified smooth} ($C^1$ SS) structure. This structure will be defined in Section \ref{c1ss}.

\begin{theorem}\label{admitthm}\cite[Theorem 1.1.1]{gluingme}
Let $(M^{2n},\omega,J)$ be a $2n$-dimensional symplectic manifold with tame almost complex structure $J$. Let $\overline{\mathcal{M}}_{0,k}(A,J)$ be the compact space of nodal $J$-holomorphic genus zero stable maps in class $A$ with $k$ marked points modulo reparametrization. Let $d=2n+2c_1(A)+2k-6$. Then $X:=\overline{\mathcal{M}}_{0,k}(A,J)$ has an oriented, $d$-dimensional, weak $C^1$ SS Kuranishi atlas $\mathcal{K}$.
\end{theorem}

Moreover, the charts $U$ of $\mathcal{K}$ (described in Section \ref{gwcharts}) lie over Deligne-Mumford space $\overline{\mathcal{M}}_{0,k}$ in the sense that there is a forgetful map
\begin{equation}\label{forgetful}
\pi_0:U\rightarrow \overline{\mathcal{M}}_{0,k}.
\end{equation}
It is shown in \cite{gluingme} that $\overline{\mathcal{M}}_{0,k}$ admits a $C^1$ SS structure compatible with the atlas $\mathcal{K}$ and its usual topology. The main result is the following.

\begin{proposition}\cite[Proposition 1.1.2]{gluingme}\label{dm1}
Genus zero Deligne-Mumford space $\overline{\mathcal{M}}_{0,k}$ admits the structure of a $C^1$ SS manifold, denoted $\overline{\mathcal{M}}_{0,k}^{new}$, that is compatible with the $C^1$ SS atlas of Theorem \ref{admitthm} in the following sense. Let $\mathcal{K}$ denote the Kuranishi atlas on the Gromov-Witten moduli space $\overline{\mathcal{M}}_{0,k}(A,J)$ from Theorem \ref{admitthm} and $U$ a domain of $\mathcal{K}$. Then the forgetful map
\[\pi_0:U\rightarrow \overline{\mathcal{M}}_{0,k}^{new}\]
is $C^1$ SS.
\end{proposition}

The first result of this paper is that a $C^1$ SS structure is sufficient to obtain a virtual fundamental class for any virtual dimension. This generalizes the same result for \emph{smooth} Kuranishi atlases, Theorem \ref{theoremb} below.

\begin{theorem}\label{c1thm}
Let $\mathcal{K}$ be a oriented, weak, d dimensional $C^1$ SS Kuranishi atlas on a compact metrizable space $X$. Then $\mathcal{K}$ determines a cobordism class of oriented, compact weighted branched $C^1$ manifolds, and an element $[X]_\mathcal{K}^{vir}\in \check{H}_d(X;\mathbb{Q})$. Both depend only on the oriented cobordism class of $\mathcal{K}$. Here $\check{H}_{*}$ denotes $\check{C}$ech homology.
\end{theorem}

Combining Theorems \ref{admitthm} and \ref{c1thm} we see that the Gromov-Witten moduli space admits a virtual fundamental cycle (for any virtual dimension). We can use this to define genus zero Gromov-Witten invariants with homological constraints. Let $X:=\overline{\mathcal{M}}_{0,k}(A,J)$ be the compact space of nodal $J$-holomorphic genus zero stable maps in class $A\in H_2(M)$ with $k$ marked points modulo reparametrization. Let $c_1,\ldots,c_k\in H_{*}(M)$ and $\beta\in H_{*}(\overline{\mathcal{M}}_{0,k})$. Let $Z_{c_1\times\cdots\times c_k}, Z_{\beta}$ be cycle representatives of $c_1\times\cdots\times c_k, \beta$. The space $X$ carries an evaluation map
\[ev_k:X\rightarrow M^k\]
and a forgetful map
\[\pi_0:X\rightarrow\overline{\mathcal{M}}_{0,k}.\]
The Gromov-Witten invariant $GW_{A,k}^M(c_1,\ldots, c_k;\beta)$ is meant to count the number of curves in $X$ with domain in $Z_\beta$ and with marked points lying on $Z_{c_1\times\cdots\times c_k}$. We can define this precisely as follows. Theorems \ref{admitthm} and \ref{c1thm} allows us to define the virtual fundamental class $[X]_\mathcal{K}^{vir}\in \check{H}_{*}(X;\mathbb{Q})$. We can then define the Gromov-Witten invariant
\begin{equation}\label{gwdef}
GW_{A,k}^M(c_1,\ldots, c_k) = (ev_k\times\pi_0)_{*}([X]_\mathcal{K}^{vir})\cdot (c_1\times\cdots\times c_k\times\beta) \in \mathbb{Q}
\end{equation}
using the intersection product in $M^k\times\overline{\mathcal{M}}_{0,k}$.

In \cite{axioms} Kontsevich and Manin listed axioms\footnote{These ``axioms" are not expected to characterize Gromov-Witten invariants uniquely.} for Gromov-Witten invariants. These axioms are meant to be the main properties of Gromov-Witten invariants; for instance, they are sufficient to prove Kontsevich's celebrated formula for the Gromov-Witten invariants of $\mathbb{C}P^2$. They considered invariants of all genera, but we will only consider genus zero invariants.\footnote{Among the reasons we restrict to the genus zero case is that genus zero Deligne-Mumford space is a smooth manifold with coordinates given by cross ratios. One expects the results of this paper and \cite{gluingme} to hold in the higher genus case, but further work is required. These issues are discussed more in \cite{gluingme}.} The second main result of this paper is that Gromov-Witten invariants, as defined using Kuranishi atlases, satisfy these axioms.

\begin{theorem}\label{axiomthm}
Let $M^{2n}$ be a closed symplectic manifold and $A\in H_{2}(M),$ $c_i\in H_{*}(M),$  $\beta\in H_{*}(\overline{\mathcal{M}}_{0,k})$. Let $GW_{A,k}^M(c_1,\ldots,c_k;\beta)$ be the Gromov-Witten invariant, defined by (\ref{gwdef}) using Kuranishi atlases. Then these invariants satisfy the Gromov-Witten axioms of Kontsevich and Manin \cite{axioms} listed below.
\begin{enumerate}
\item[\textbf{(Effective)}] If $\omega(A) < 0$, then $GW^M_{A,k}=0$.
\item[\textbf{(Symmetry)}] For each permutation $\sigma\in S_k$,
\[GW^{M}_{A,k}(c_{\sigma(1)},\ldots,c_{\sigma(k)};\sigma_{*}\beta) = \varepsilon(\sigma;c)GW^M_{A,k}(c_1,\ldots,c_k;\beta)\]
where
\[\varepsilon(\sigma;c):=(-1)^{\#\{i<j ~|~ \sigma(i)>\sigma(j), \deg(c_i)\deg(c_j)\in2\mathbb{Z}+1\}}\]
denotes the sign of the induced permutation on the classes of odd degree and $\sigma_*\beta$ denotes the pushforward of a homology class $\beta\in H_*(\overline{\mathcal{M}}_{0,k})$ under the diffeomorphism of $\overline{\mathcal{M}}_{0,k}$ defined as follows:
Identify $\overline{\mathcal{M}}_{0,k}$ with its image in $(S^2)^N$ using the extended cross ratio function $\mathbf{z}\mapsto w_{ijkl}(\mathbf{z})$, and denote its elements by tuples $\{w_{ikjl}\}$ (c.f. \cite[Appendix D]{JHOL}). Then $\sigma$ defines a diffeomorphism via
\[\{w_{i_0i_1i_2i_3}\}\mapsto \{w_{\sigma(i_0)\sigma(i_1)\sigma(i_2)\sigma(i_3)}\}.\]
\item[\textbf{(Grading)}] If $GW^M_{A,k}(c_1,\ldots,c_k;\beta)\neq0$, then
\[\sum_{i=1}^{k}(2n-\deg(c_i)) - \deg(\beta) = 2n + 2c_1(A).\]
\item[\textbf{(Homology)}] For every $A\in H_2(M;\mathbb{Z})$ and every integer $k\geq 3$ there exists a homology class
\[\sigma_{A,k}\in H_{2n+2c_1(A)+2k-6}(M^k\times\overline{\mathcal{M}}_{0,k})\]
such that
\[GW^M_{A,k}(c_1,\ldots,c_k;\beta) = (c_1\times\ldots\times c_k\times\beta)\cdot\sigma_{A,k}.\]
\item[\textbf{(Fundamental class)}] Let
\[\pi_{0,k}:\overline{\mathcal{M}}_{0,k}\rightarrow\overline{\mathcal{M}}_{0,k-1}\]
denote the map that forgets the last marked point. Then
\[GW^M_{A,k}(c_1,\ldots,c_{k-1},[M];\beta) = GW^M_{A,k-1}(c_1,\ldots,c_{k-1};(\pi_{0,k})_*\beta).\]
In particular, this invariant vanishes if $\beta=[\overline{\mathcal{M}}_{0,k}]$.
\item[\textbf{(Divisor)}] If $(A,k)\neq(0,3)$ and $\deg(c_k)=2n-2$, then
\[GW^M_{A,k}\Big(c_1,\ldots,c_k;PD\big(\pi_{0,k}^*PD(\beta)\big)\Big) = (c_k\cdot A)~GW^M_{A,k-1}(c_1,\ldots,c_{k-1};\beta).\]
\item[\textbf{(Zero)}] If $A=0$, then $GW^M_{0,k}(c_1,\ldots,c_k;\beta) = 0$ whenever $\deg(\beta)>0$, and
\[GW_{0,k}^M(c_1,\ldots,c_k;[\textnormal{pt}]) = c_1\cap\cdots\cap c_k\]
where [pt] denotes the homology class of a point.
\item[\textbf{(Splitting)}]Fix a basis $e_0,\ldots,e_N$ of the homology $H_*(M)$, let
\[g_{\nu\mu} := e_\nu\cdot e_\mu,\]
and denote by $g^{\nu\mu}$ the inverse matrix. Fix a partition of the index set $\{1,\ldots,k\} = S_0\sqcup S_1$ such that $k_i:=|S_i|\geq 2$ for $i=0,1$ and denote by
\[\phi_{S}:\overline{\mathcal{M}}_{0,k_0+1}\times\overline{\mathcal{M}}_{0,k_1+1}\rightarrow\overline{\mathcal{M}}_{0,k}\]
the canonical map which identifies the last marked point of a stable curve in $\overline{\mathcal{M}}_{0,k_0+1}$ with the first marked point of a stable curve in $\overline{\mathcal{M}}_{0,k_0+1}$. The remaining indices have the unique ordering such that the relative order is preserved, the first $k_0$ points in $\overline{\mathcal{M}}_{0,k_0+1}$ are mapped to the points indexed by $S_0$, and the last $k_1$ points in $\overline{\mathcal{M}}_{0,k_1+1}$ are mapped to the points indexed by $S_1$. Fix two homology classes $\beta_0\in H_*(\overline{\mathcal{M}}_{0,k_0+1})$ and $\beta_1\in H_*(\overline{\mathcal{M}}_{0,k_1+1})$. Then
\begin{align*}
&GW^M_{A,k}(c_1,\ldots,c_k;{\phi_S}_{*}(\beta_0\otimes\beta_1))\\
= &\varepsilon(S,c)\sum_{A_0+A_1=A}\sum_{\nu,\mu}GW^M_{A_0,k_0+1}(\{c_i\}_{i\in S_0},e_\nu;\beta_0)g^{\nu\mu}GW^M_{A_1,k_1+1}(e_\mu,\{c_j\}_{j\in S_1};\beta_1),
\end{align*}
where
\[\varepsilon(S,c):=(-1)^{\#\{j<i|i\in S_0,j\in S_1,\deg(a_i)\deg(a_j)\in 2\mathbb{Z}+1\}}.\]
\end{enumerate}
\end{theorem}

\begin{remark}
The question of how to construct a virtual fundamental class for moduli spaces of $J$-holomorphic curves has a long history. One can refer to \cite[Remark 1.1.3]{gluingme} and the references therein for a discussion of how this project relates to other methods. We mention that Fukaya-Ono \cite{fo} prove the Gromov-Witten axioms, except the \textit{(Homology)} axiom. Li-Tian \cite{litianalg} prove the Gromov-Witten axioms in the case that $M$ is a projective variety.\hfill$\Diamond$
\end{remark}

Section \ref{gwatlases} outlines the construction of the virtual fundamental cycle for Gromov-Witten moduli spaces. Section \ref{atlases} contains the necessary background on Kuranishi atlases and the construction of the virtual fundamental cycle associated to a Kuranishi atlas. One of the steps in this construction is building a \textit{perturbation section}; this discussion is somewhat more involved as is deferred until Section \ref{sections}. In Section \ref{originalsection} we describe the original construction of McDuff-Wehrheim in some detail and then in Sections \ref{sectionsc1}, \ref{tsections}, \ref{complementsections} show how to adapt the original construction to different circumstances that are required in this paper. Constructing these sections is the most technically challenging component of the proofs of the main theorems of this paper and requires understanding the original construction. Section \ref{c1ss} gives a review of stratified spaces, discusses their role with respect to Kuranishi atlases, and proves Theorem \ref{c1thm} modulo results from Section \ref{sectionsc1}. Section \ref{gwcharts} describes the construction of the Kuranishi atlases on $X=\overline{\mathcal{M}}_{0,k}(A,J)$ in Theorem \ref{admitthm}. Section \ref{constraints} completes the construction of Gromov-Witten invariants by considering homological constraints and proves basic results about them: Section \ref{homconstruction} discusses constraints on the image of curves and Section \ref{domainconstruction} discusses constraints on the domains. Section \ref{tsubatlases} contains the details regarding these constrained invariants and introduces the key notion of a \textit{transverse subatlas}; this section contains the structural ideas and geometric arguments with some more technical discussions deferred to Section \ref{tsections}. Theorem \ref{axiomthm} is proved in Section \ref{gwaxioms}; this section contains the geometric arguments and uses the results of Sections \ref{constraints} and \ref{sections}.\\

\noindent\textbf{Acknowledgements}: The author is very thankful to his advisor Dusa McDuff for many useful discussions on Kuranishi atlases as well as for suggesting this project.

\section{Gromov-Witten virtual fundamental cycle}\label{gwatlases}
This section provides background on Kuranishi atlases, stratified smooth spaces, and applies these theories to Gromov-Witten moduli spaces.

\subsection{Kuranishi atlases}\label{atlases}
First, we briefly recall the basic definitions of Kuranishi charts and the main theorems.

\begin{definition}\label{chart}
Let $F\subset X$ be a nonempty open subset. A \textbf{Kuranishi chart} for $X$ with footprint $F$ is a tuple $\mathbf{K} = (U,E,\Gamma,s,\psi)$ consisting of
\begin{itemize}
\item The \textbf{domain} $U$, which is a finite dimensional manifold.
\item The \textbf{obstruction space} $E$, which is a finite dimensional vector space.
\item The \textbf{isotropy group} $\Gamma$, which is a finite group acting smoothly on $U$ and linearly on $E$.
\item The \textbf{section} $s:U\rightarrow E$ which is a smooth $\Gamma$-equivariant map.
\item The \textbf{footprint map} $\psi:s^{-1}(0)\rightarrow X$ which is a continuous map that induces a homeomorphism
\[\underline{\psi}:\faktor{s^{-1}(0)}{\Gamma}\rightarrow F\]
where $F\subset X$ is an open set called the \textbf{footprint}.
\end{itemize}
The \textbf{dimension} of $\mathbf{K}$ is $\dim{\mathbf{K}} = \dim U - \dim E$.\\
\end{definition}

To implement the topological constructions needed on Kuranishi charts, it will be convenient to consider the notion of intermediate Kuranishi charts.
\begin{definition}\label{intermediatedef}
The \textbf{intermediate chart} $\underline{\mathbf{K}}:=(\underline{U},\underline{U\times E},\underline{s},\underline{\psi})$ associated to a Kuranishi chart $\mathbf{K}=(U,E,\Gamma,s,\psi)$ consists of
\begin{itemize}
\item The \textbf{intermediate domain} $\underline{U}:=\faktor{U}{\Gamma}$.
\item The \textbf{intermediate obstruction ``bundle''}, whose total space is the quotient $\underline{U\times E}$ by the diagonal action of $\Gamma$. This carries a projection $pr:\underline{E}\rightarrow\underline{U}$ and a zero section $0:\underline{U}\rightarrow\underline{E}$.
\item The \textbf{intermediate section} $\underline{s}:\underline{U}\rightarrow\underline{E}$ induced by $s$.
\item The \textbf{intermediate footprint map} $\underline{\psi}:\underline{s}^{-1}(\textnormal{im }0)\rightarrow X$ induced by $\psi$. Note that \textnormal{im}$\underline{\psi} = F\subset X$.
\end{itemize}
We will let $\pi:U\rightarrow\underline{U}$ denote the quotient map.
\end{definition}

Suppose that we have a finite collection of Kuranishi charts $(\mathbf{K}_i)_{i=1,\ldots,N}$ such that for each $I\subset\{1,\ldots,N\}$ satisfying $F_I:=\bigcap_{i\in I}F_i\neq\emptyset$, we have a Kuranishi charts $\mathbf{K}_I$ with
\begin{itemize}
\item Footprint $F_I$,
\item Isotropy group $\Gamma_I:=\prod_{i\in I}\Gamma_i$,
\item Obstruction space $E_I:=\prod_{i\in I}E_i$ on which $\Gamma_I$ acts with the product action.
\end{itemize}
Such charts $\mathbf{K}_I$ are known as \textbf{sum charts}. Thus, for $I\subset J$ we have a natural splitting $\Gamma_J\cong\Gamma_I\times\Gamma_{J\setminus I}$ and an induced projection $\rho_{IJ}^\Gamma:\Gamma_J\rightarrow\Gamma_I$. Moreover, we have a natural inclusion $\widehat{\phi}_{IJ}:E_I\rightarrow E_J$ which is equivariant with respect to the inclusion $\Gamma_I\hookrightarrow\Gamma_J$. Thus we can consider $E_I$ as a subset of $E_J$.

\begin{definition}\label{transitiondef}
Let $X$ be a compact metrizable space.
\begin{itemize}
\item A \textbf{covering family of basic charts} for $X$ is a finite collection $(\mathbf{K}_i)_{i=1,\ldots,N}$ of Kuranishi charts for $X$ whose footprints cover $X$.
\item \textbf{Transition data} for $(\mathbf{K}_i)_{i=1,\ldots,N}$ is a collection of sum charts $(\mathbf{K}_J)_{J\subset\mathcal{I}_\mathcal{K},|J|\geq 2}$, and coordinate changes $(\widehat{\Phi}_{IJ})_{I,J\in\mathcal{I}_\mathcal{K}, I\subsetneq J}$ satisfying:
    \begin{enumerate}
    \item $\mathcal{I}_\mathcal{K} = \left\{I\subset\{1,\ldots,N\} | \bigcap_{i\in I}F_i \neq 0\right\}$.
    \item $\widehat{\Phi}_{IJ}$ is a coordinate change $\mathbf{K}_I\rightarrow\mathbf{K}_J$. For the precise definition of a coordinate change, see \cite{mwiso}. A coordinate change consists of
        \begin{enumerate}[(i)]
        \item A choice of domain $\underline{U}_{IJ}\subset \underline{U}_I$ such that $\mathbf{K}_I|_{\underline{U}_{IJ}}$ has footprint $F_J$.
        \item A $\Gamma_J$-equivariant submanifold $\widetilde{U}_{IJ}\subset U_J$ on which $\Gamma_{J\setminus I}$ acts freely. Let $\widetilde{\phi}_{IJ}$ denote the $\Gamma_J$-equivariant inclusion.
        \item A \textbf{group covering} $(\widetilde{U}_{IJ},\Gamma,\rho_{IJ},\rho_{IJ}^\Gamma)$ of $(U_{IJ},\Gamma_I)$ where $U_{IJ}:=\pi^{-1}(\underline{U}_{IJ})\subset U_I$, meaning that $\rho_{IJ}:\widetilde{U}_{IJ}\rightarrow U_{IJ}$ is the quotient map $\widetilde{U}_{IJ}\rightarrow\faktor{\widetilde{U}_{IJ}}{\ker\rho_{IJ}^\Gamma}$ composed with a diffeomorphism $\faktor{\widetilde{U}_{IJ}}{\ker\rho_{IJ}^\Gamma}\cong U_{IJ}$ that is equivariant with respect to the induced $\Gamma_I$ action on both spaces. In particular, this implies that $\underline{\rho}:\underline{\widetilde{U}}_{IJ}\rightarrow\underline{U}_{IJ}$ is a homeomorphism (see \cite[Lemma 2.1.5]{mwiso}).
        \item The linear equivariant injection $\widehat{\phi}_{IJ}:E_I\rightarrow E_J$ as above.
    \item The inclusions $\widetilde{\phi}_{IJ},\widehat{\phi}_{IJ}$ and the covering $\rho_{IJ}$ intertwine the sections and the footprint maps in the sense that
        \begin{align*}
        s_J\circ\widetilde{\phi}_{IJ} &= \widehat{\phi}_{IJ}\circ s_I\circ\rho_{IJ}\quad& &\textnormal{on }\widetilde{U}_{IJ},\\
        \psi_J\circ\widetilde{\phi}_{IJ} &= \psi_I\circ\rho_{IJ}& &\textnormal{on }s_J^{-1}(0)\cap\widetilde{U}_{IJ} = \rho_{IJ}^{-1}(s_I^{-1}(0)).
        \end{align*}
    \item The map $s_{IJ}:=s_I\circ\rho_{IJ}$ is required to satisfy an \textbf{index condition}. This ensures that any two charts that are related by a coordinate change have the same dimension. It also implies that $\widetilde{U}_{IJ}$ is an open subset of $s_J^{-1}(E_I)$.
        \end{enumerate}
    \end{enumerate}
\end{itemize}
\end{definition}
One needs to express a way in which coordinate changes are compatible. This is described by cocycle conditions. For charts $\mathbf{K}_\alpha$, where $\alpha=I,J,K$ with $I\subset J\subset K$, a \textbf{weak cocycle condition} / \textbf{cocycle condition} / \textbf{strong cocycle condition} describe to what extent the composition $\widehat{\Phi}_{JK}\circ\widehat{\Phi}_{IJ}$ and the coordinate change $\widehat{\Phi}_{IK}$ agree. These conditions require them to agree on increasingly large subsets. See \cite{mwiso} for a complete discussion of cocycle conditions.

In constructions, one can achieve a weak cocycle condition, while a strong cocycle condition is what is required for the construction of a virtual fundamental class.
\begin{definition}
A \textbf{weak Kuranishi atlas of dimension d} on a compact metrizable space $X$ is transition data
\[\mathcal{K} = (\mathbf{K}_I,\Phi_{IJ})_{I,J\in\mathcal{I}_\mathcal{K},I\subsetneq J}\]
for a covering family $(\mathbf{K}_i)_{i=1,\ldots,N}$ of dimension d for $X$ that satisfies a weak cocycle condition. Similarly a \textbf{(strong) atlas} is required to satisfy a (strong) cocycle condition.
\end{definition}
The main theorem regarding Kuranishi atlases is the following.
\begin{theorem}\label{theoremb}
Let $\mathcal{K}$ be a oriented, weak, d dimensional smooth Kuranishi atlas on a compact metrizable space $X$. Then $\mathcal{K}$ determines a cobordism class of oriented, compact weighted branched topological manifolds, and an element $[X]_\mathcal{K}^{vir}\in \check{H}_d(X;\mathbb{Q})$. Both depend only on the oriented cobordism class of $\mathcal{K}$. Here $\check{H}_{*}$ denotes $\check{C}$ech homology.
\end{theorem}
In the case of trivial isotropy groups, this is \cite[Theorem B]{mwfund}. For nontrivial isotropy groups see \cite[Theorem A]{mwiso}.

This paper will be primarily interested in the case $X = \overline{\mathcal{M}}_{0,k}(A,J)$, the compact space of nodal $J$-holomorphic genus zero stable maps on a symplectic manifold in homology class $A$ with $k$ marked points modulo reparametrization, where $J$ is a tame almost complex structure. We will also be interested in subsets $X_c\subset X$ obtained from imposing homological constraints on elements of $X$. Section \ref{gwcharts} describes an atlas on $X$ and Section \ref{homconstruction} describes an atlas on $X_c$.

For the rest of this section we give an overview of the construction of the virtual fundamental class. The construction of $[X]_{\mathcal{K}}^{vir}$ takes place in several steps: First, a weak atlas is \emph{tamed}, which is a procedure that shrinks the domains and implies desirable topological properties. The taming procedure allows us to define two categories: The \textit{domain category} $\mathbf{B}_{\mathcal{K}}$ and the \textit{obstruction category} $\mathbf{E}_\mathcal{K}$, which are built from the domains $U_I$ and the obstruction bundles \mbox{$U_I\times E_I\rightarrow U_I$} respectively. These categories are equipped with a projection functor \mbox{pr$:\mathbf{E}_{\mathcal{K}}\rightarrow\mathbf{B}_{\mathcal{K}}$} and a section functor $s:\mathbf{B}_{\mathcal{K}}\rightarrow\mathbf{E}_{\mathcal{K}}$ that come from pr$_I:U_I\times E_I\rightarrow U_I$ and $s_I:U_I\rightarrow U_I\times E_I$. However, this category has too many morphisms so we pass to a full subcategory $\mathbf{B}_{\mathcal{K}}|_{\mathcal{V}}$ of $\mathbf{B}_{\mathcal{K}}$ known as a \emph{reduction} of $\mathcal{K}$. In the reduction, one can construct a perturbation functor $\nu:\mathbf{B}_{\mathcal{K}}|_{\mathcal{V}}\rightarrow\mathbf{E}_{\mathcal{K}}|_{\mathcal{V}}$. Then for an appropriately constructed perturbation $\nu$, the realization $|(s|_{\mathcal{V}}+\nu)^{-1}(0)|$ is a compact weighted branched manifold. For background on weighted branched manifolds see \cite[Appendix A]{mwiso}. The virtual fundamental cycle is then constructed from this zero set. Relevant definitions and a summary of these constructions is given below.

\noindent\\ \textbf{Taming}: The first step in the construction of the virtual fundamental cycle is taming. This procedure is topological in nature. We now give the basic definitions and results regarding taming.
\begin{definition}
Let $\{F_i\}$ be a finite open cover of a compact space $X$. We say that $\{F_i'\}$ is a \textbf{shrinking} of $\{F_i\}$ if $F_i'\sqsubset F_i$ are precompact open sets, $\{F_i'\}$ is still an open cover of $X$, and
\begin{equation}\label{shrinkingeq}
F_I:=\bigcap_{i\in I}F_i \neq \emptyset \quad\Rightarrow\quad F_I':=\bigcap_{i\in I}F_i'\neq \emptyset.
\end{equation}
\end{definition}
One can always find a shrinking of an open cover of a compact Hausdorff space. We can also define a shrinking of a Kuranishi atlas.
\begin{definition}\label{kshrinking}
Let $\mathcal{K}=(\mathbf{K}_I,\widehat{\Phi}_{IJ})_{I,J\in\mathcal{I}_\mathcal{K}, I\subsetneq J}$ be a weak Kuranishi atlas. We say that a weak Kuranishi atlas $\mathcal{K}'=(\mathbf{K}_I',\widehat{\Phi}_{IJ}')_{I,J\in\mathcal{I}_\mathcal{K}', I\subsetneq J}$ is a \textbf{shrinking} of $\mathcal{K}$ if
\begin{enumerate}[(i)]
\item The footprint cover $\{F_i'\}$ is a shrinking of the cover $\{F_i\}$. In particular, $\mathcal{I}_\mathcal{K}=\mathcal{I}_{\mathcal{K}'}$.
\item For each $I\in\mathcal{I}_\mathcal{K}$ the chart $\mathbf{K}_I'$ is a restriction of $\mathbf{K}_I$ to a precompact domain.
\item For $I,J\in\mathcal{I}_\mathcal{K}, I\subsetneq J$, the coordinate change $\widehat{\Phi}_{IJ}'$ is a restriction of $\widehat{\Phi}_{IJ}$.
\end{enumerate}
We write $\mathcal{K}'\sqsubset\mathcal{K}$.
\end{definition}
\begin{definition}
A weak Kuranishi atlas is called \textbf{tame} if for all $I,J,K\in\mathcal{I}_{\mathcal{K}}$ with $I\subset J\subset K$, we have
\begin{enumerate}
\item $\underline{U}_{IJ}\cap \underline{U}_{IK} = \underline{U}_{I(J\cup K)}$.
\item $(\widetilde{\underline{\phi}}_{IJ}\circ \rho_{IJ}^{-1})(\underline{U}_{IK}) = \underline{U}_{JK}\cap \underline{s}_J^{-1}(\underline{U_J\times\widehat{\phi}_{IJ}(E_I)}).$
\end{enumerate}
\end{definition}
\begin{proposition}\cite[Proposition 3.3.5]{mwtop}\label{taming}
Let $\mathcal{K}$ be a weak Kuranishi atlas. Then there is a shrinking $\mathcal{K}'\sqsubset\mathcal{K}$ such that $\mathcal{K}'$ is tame. Moreover, if $\{F_i\}$ are the footprints of $\mathcal{K}$ and $\{F_i'\}$ is a shrinking of $\{F_i\}$, then $\mathcal{K}'$ can be chosen to have footprints $\{F_i'\}$.
\end{proposition}

The procedure of proving Proposition \ref{taming} is known as taming and is topological in nature. In fact, Proposition \ref{taming} is proved by considering the intermediate atlas $\underline{\mathcal{K}}$ (see Definition \ref{intermediatedef}), which is shown in \cite{mwtop} to have the structure of a \textbf{filtered weak topological atlas}. This is a notion of an atlas that allows for more general domains than smooth manifolds (in particular group quotients). A taming for $\mathcal{K}$ is then achieved by lifting a taming of $\underline{\mathcal{K}}$.

\noindent\\ \textbf{Reduction}: The next step in the construction of the virtual fundamental cycle is the process of \emph{reduction}. Before dealing with this process, we need some more preliminaries.
\begin{definition}\label{catdef}
Given a Kuranishi atlas $\mathcal{K}$ define the \textbf{domain category} $\mathbf{B}_\mathcal{K}$ to have object space
\[\textnormal{Obj}_{\mathbf{B}_\mathcal{K}} := \bigsqcup_{I\in\mathcal{I}_\mathcal{K}}U_I\]
and morphism space
\[\textnormal{Mor}_{\mathbf{B}_\mathcal{K}}:=\bigsqcup_{I,J\in\mathcal{I}_\mathcal{K},I\subset J}\widetilde{U}_{IJ}\times \Gamma_I.\]
Here we use the convention $\widetilde{U}_{II}:=U_I$.

The source and target of morphisms are given by
\[(I,J,x,\gamma)\in \textnormal{Mor}_{\mathbf{B}_\mathcal{K}}\Big(\big(I,\gamma^{-1}\rho_{IJ}(x)\big),\big(J,\widetilde{\phi}_{IJ}(x)\big)\Big).\]
Composition is given by
\[(I,J,x,\gamma)\circ(J,K,y,\delta):=\big(I,K,z:=\widetilde{\phi}^{-1}_{IK}\big(\widetilde{\phi}_{JK}(z)\big),\rho_{IJ}^\Gamma(\delta)\gamma\big)\]
where defined.
\end{definition}
We can analogously define the \textbf{obstruction category} $\mathbf{E}_\mathcal{K}$ with objects
\[\textnormal{Obj}_{\mathbf{E}_\mathcal{K}} := \bigsqcup_{I\in\mathcal{I}_\mathcal{K}}U_I\times E_I\]
and morphisms
\[\textnormal{Mor}_{\mathbf{B}_\mathcal{K}}:=\bigsqcup_{I,J\in\mathcal{I}_\mathcal{K},I\subset J}\widetilde{U}_{IJ}\times E_I\times\Gamma_I.\]
There are functors $pr_\mathcal{K}:\mathbf{E}_\mathcal{K}\rightarrow\mathbf{B}_\mathcal{K}, s_\mathcal{K}:\mathbf{B}_\mathcal{K}\rightarrow\mathbf{E}_\mathcal{K}$ obtained from the projection map \mbox{pr$_I:U_I\times E_I\rightarrow U_{I}$} and the section $s_I:U_I\rightarrow U_I\times E_I$. There is also the full subcategory of zero sets $\bigsqcup_{I\in\mathcal{I}_{\mathcal{K}}}\{I\}\times s_I^{-1}(0)\subset \textnormal{Obj}_{\mathbf{B}_\mathcal{K}}$.

We can form the \textbf{topological realization of the category} $\mathbf{B}_\mathcal{K}$, $|\mathbf{B}_\mathcal{K}|$ or just $|\mathcal{K}|$, which is the space formed from the quotient of the objects by the equivalence relation generated by the morphisms. We denote this quotient by the maps $\pi_\mathcal{K}:U_I\rightarrow |\mathcal{K}|$.

The topological realization $|\mathcal{K}|$ can be quite wild, but the taming procedure gives desirable topological properties. For example, if $\mathcal{K}$ is tame, then $|\mathcal{K}|$ is Hausdorff and $\displaystyle\pi_{\underline{\mathcal{K}}}:\faktor{U_I}{\Gamma_I}\rightarrow|\mathcal{K}|$ is a homeomorphism onto its image (see \cite[Theorem 3.1.9]{mwtop}). The full subcategory of zero sets has a realization $|s_\mathcal{K}|^{-1}(0)\subset|\mathcal{K}|$ that is naturally homeomorphic to $X$.

The taming condition also allows one to put a metric on $|\mathcal{K}|$. Suppose $\mathcal{K}, \mathcal{K}'$ are tame atlases such that $\mathcal{K}\sqsubset\mathcal{K}'\sqsubset\mathcal{K}''$ (recall from Definition \ref{kshrinking} that $\sqsubset$ denotes shrinking and means we have a precompact inclusion of domains). Then one can put a metric $d$ on $|\mathcal{K}|$ such that the pullback metric $\underline{d}_I:=(\underline{\pi}_\mathcal{K}|_{\underline{U}_I})^{*}d$ on $\underline{U}_I$ induces the given topology on $\underline{U}_I$. In this situation we call $d$ an \textit{admissible metric}. See \cite[Lemma 3.1.8 and Theorem 3.1.9]{mwtop} for more details regarding metrics.

The above properties are a consequence of the fact that the taming condition simplifies the equivalence relation defining $|\mathcal{K}|$. Specifically, for a tame atlas $\mathcal{K}$, $(I,x)\sim(J,y)$ in $|\mathcal{K}|$ if and only if there are morphisms $(I,x)\rightarrow(I\cup J, z)\leftarrow (J,y)$ for some $z\in U_{I\cup J}$. The reduction procedure further simplifies the equivalence relation.

\begin{definition}\label{reductiondef}
A \textbf{reduction} of a tame Kuranishi atlas $\mathcal{K}$ consists of a tuple of open subsets $V_I\subset U_I$ satisfying:
\begin{enumerate}
\item\label{reduction0} $V_I=\pi^{-1}(\underline{V}_I)$ for each $I\in\mathcal{I}_{\mathcal{K}}$. Hence $V_I$ is $\Gamma_I$-invariant.
\item\label{reduction1} $V_I\sqsubset U_I$, that is $V_I$ is precompact in $U_I$, for all $I\in\mathcal{I}_\mathcal{K}$, and if $V_I\neq\emptyset$, then $V_I\cap s^{-1}_I(0)\neq\emptyset$.
\item\label{reduction2} $\pi_{\mathcal{K}}(\overline{V_I})\cap\pi_{\mathcal{K}}(\overline{V_J})\neq\emptyset$ implies either $I\subset J$ or $J\subset I$.
\item\label{reduction3} The zero set $|s_\mathcal{K}|^{-1}(0)$ is contained in $\bigcup_{I\in\mathcal{I}_{\mathcal{K}}}\pi_{\mathcal{K}}(V_I).$
\end{enumerate}

\begin{proposition}\cite[Proposition 5.3.5]{mwtop}\label{reduction}
Every tame Kuranishi atlas $\mathcal{K}$ has a reduction $\mathcal{V}$.
\end{proposition}
\end{definition}

Similar to the taming procedure, the reduction procedure can done for tame topological atlases. Proposition \ref{reduction} is proved by constructing a reduction on the intermediate atlas $\underline{\mathcal{K}}$ and then lifting to $\mathcal{K}$.

\noindent\\ \textbf{Perturbation section}: Given a reduction $\mathcal{V}:=(V_I)_{I\in\mathcal{I}_{\mathcal{K}}}$, we define the \textbf{reduced domain category} $\mathbf{B}_{\mathcal{K}}|_{\mathcal{V}}$ and the \textbf{reduced obstruction category} $\mathbf{E}_{\mathcal{K}}$ to be the full subcategories of $\mathbf{B}_{\mathcal{K}}$ and $\mathbf{E}_{\mathcal{K}}$ with objects $\sqcup_{I\in\mathcal{I}_\mathcal{K}}V_I$ and $\sqcup_{I\in\mathcal{I}_\mathcal{K}}V_I\times E_I$ respectively. Denote $s_{\mathcal{K}}|_{\mathcal{V}}:\mathbf{B}_{\mathcal{K}}|_{\mathcal{V}}\rightarrow \mathbf{E}_{\mathcal{K}}|_{\mathcal{V}}$ the restriction of $s_{\mathcal{K}}$. One would want to find a transverse perturbation functor \mbox{$\nu:\mathbf{B}_{\mathcal{K}}|_{\mathcal{V}}\rightarrow \mathbf{E}_{\mathcal{K}}|_{\mathcal{V}}$} and construct a manifold from the perturbed zero set $(s_{\mathcal{K}}|_{\mathcal{V}}+\nu)^{-1}(0)$. In the absence of isotropy, this is precisely what is done in \cite{mwfund}. In the presence of isotropy, more care is required. The case with isotropy is done in \cite{mwiso}, while the paper \cite{mcdufforbifold} describes the easier case when $X$ is an orbifold and is helpful in illuminating the construction. Following \cite{mwiso}, to deal with isotropy one reduces the problem to the case without isotropy and then appeals to the construction of \cite{mwfund}. To do this, one considers the \textbf{pruned categories} $\mathbf{B}_{\mathcal{K}}|_{\mathcal{V}}^{\setminus\Gamma}$ and $\mathbf{E}_{\mathcal{K}}|_{\mathcal{V}}^{\setminus\Gamma}$. Roughly speaking, these are categories obtained from forgetting the morphisms coming from the isotropy group action so that the only remaining morphisms are those coming from the projections $\rho_{IJ}$; then the pruned categories have trivial isotropy. For precise definitions and statements regarding pruned categories, one can refer to \cite{mwiso}. One then constructs a perturbation functor on the level of the pruned categories. The correct notion of a perturbation functor is the following.

\begin{definition}[\cite{mwfund,mwiso}]\label{perturbdef}
A \textbf{perturbation} of $\mathcal{K}$ is a smooth functor $\nu:\mathbf{B}_{\mathcal{K}}|_{\mathcal{V}}^{\setminus\Gamma}\rightarrow\mathbf{E}_{\mathcal{K}}|_{\mathcal{V}}^{\setminus\Gamma}$ between the pruned domain and obstruction categories of some reduction $\mathcal{V}$ of $\mathcal{K}$, such that $pr_{\mathcal{K}}|_{\mathcal{V}}^{\setminus\Gamma}\circ\nu = id_{\mathbf{B}_{\mathcal{K}}|_{\mathcal{V}}^{\setminus\Gamma}}$.
That is, $\nu=(\nu_I)_{I\in\mathcal{I}_{\mathcal{K}}}$ is given by a family of smooth maps $\nu_I:V_I\rightarrow E_I$ such that $\nu_I\circ\rho_{IJ}=\nu_J$ on $\widetilde{V}_{IJ}:=V_J\cap\rho_{IJ}^{-1}(V_I)$. We say that a perturbation $\nu$ is
\begin{itemize}
\item \textbf{admissible} if we have $d_{y}\nu_J(T_yV_J)\subset \textnormal{im}\widehat{\phi}_{IJ}$ for all $I\subsetneq J$ and $y\in\widetilde{V}_{IJ}$;
\item \textbf{transverse} if $s_I|_{V_I}+\nu_I:V_I\rightarrow E_I$ is transverse to $0$ for each $I\in\mathcal{I}_{\mathcal{K}}$;
\item \textbf{precompact} if there is a  precompact open subset $\mathcal{C}\sqsubset\mathcal{V}$, which itself is a reduction, such that
    \[\pi_{\mathcal{K}}\big(\cup_{I\in\mathcal{I}_{\mathcal{K}}}(s_I|_{V_I}+\nu_I)^{-1}(0)\big)\subset \pi_{\mathcal{K}}(\mathcal{C}).\]
\end{itemize}
\end{definition}

The existence of an admissible, transverse, precompact perturbation is proved in \cite{mwfund}. In fact, they prove the existence of a perturbation with even more refined properties. The construction of this perturbation is a delicate iterative procedure. We give the precise statement of the existence of a perturbation section in Proposition \ref{section}; there we also outline its construction.

The virtual fundamental class is then constructed from the zeros of the perturbed section as described below.

\begin{definition}
Given a perturbation $\nu$, we define the \textbf{perturbed zero set} $|\mathbf{Z}^\nu|$ to be the realization of the full subcategory $\mathbf{Z}^\nu$ of $\mathbf{B}_{\mathcal{K}}|_{\mathcal{V}}^{\setminus\Gamma}$ with object space
\[(s_{\mathcal{K}}|_\mathcal{V}^{\setminus\Gamma}+\nu)^{-1}(0).\]
Thus, it is given by local zero sets $(s_I|_{V_I}+\nu_I)^{-1}(0)$ quotiented by the morphisms of $\mathbf{B}_{\mathcal{K}}|_{\mathcal{V}}^{\setminus\Gamma}$
\end{definition}

Therefore, the existence of a suitable perturbation gives us a zero set that is cut out transversally from $\mathbf{B}_{\mathcal{K}}|_{\mathcal{V}}^{\setminus\Gamma}$. We then show (stated precisely below in Proposition \ref{wbm}) that we can add some of the isotropy morphisms back in, with the tradeoff that we add weights to the corresponding branches of the zero set resulting in the structure of a weighted branched manifold on the maximal Hausdorff quotient\footnote{A topological space $X$ has a unique quotient $|X|_{\mathcal{H}}$ called the \textbf{maximal Hausdorff quotient} which is a Hausdorff that satisfies the universal property that any continuous map from $X$ to a Hausdorff space factors through the quotient map $X\rightarrow|X|_{\mathcal{H}}$. See \cite[Lemma A.2]{mwiso}.} of the perturbed solution set $|(s_{\mathcal{K}}|_\mathcal{V}^{\setminus\Gamma}+\nu)^{-1}(0)| \subset |\mathbf{B}_{\mathcal{K}}|_{\mathcal{V}}^{\setminus\Gamma}|$. This perturbed zero set is not a subset of the realization $|\mathcal{K}|$, but its maximal Hausdorff quotient supports a fundamental class that is represented by a continuous map $|(s_{\mathcal{K}}|_\mathcal{V}^{\setminus\Gamma}+\nu)^{-1}(0)|_{\mathcal{H}}\rightarrow |\mathcal{K}|$.

\begin{proposition}\label{wbm}\cite[Theorem 3.2.8]{mwiso}
Let $\mathcal{K}$ be a tame Kuranishi atlas of dimension $d$ and let $\nu$ be an admissible, transverse, precompact perturbation of $\mathcal{K}$ with respect to nested reductions $\mathcal{C}\sqsubset\mathcal{V}$. Then the perturbed zero set $\mathbf{Z}^\nu$ can be completed to a compact $d$-dimensional weighted nonsingular branched groupoid (see \cite[Definition A.4]{mwiso}) $\mathbf{\widehat{Z}}^\nu$ with the same realization $|\mathbf{\widehat{Z}}^\nu|=|\mathbf{Z}^\nu|$. Additionally, let $|\cdot|_{\mathcal{H}}$ denote the maximal Hausdorff quotient. Then
\[\Lambda^\nu(p):=|\Gamma_I|^{-1}\#\{z\in Z_I|\pi_H(|z|)=p\},\quad p\in |Z_I|_{\mathcal{H}}\]
defines a weighting function $\Lambda^\nu:|\mathbf{Z}^\nu|_{\mathcal{H}}\rightarrow\mathbb{Q}^+$ on the maximal Hausdorff quotient of the perturbed zero set $|\mathbf{Z}^\nu|_{\mathcal{H}}$. Together $(|\widehat{\mathbf{Z}}^\nu|_{\mathcal{H}},\Lambda^\nu)$ has the structure of a compact $d$-dimensional weighted branched manifold (see \cite[Definition A.5]{mwiso}).
\end{proposition}

The virtual fundamental cycle $[X]_\mathcal{K}^{vir}$ is then constructed from $(|\widehat{\mathbf{Z}}^\nu|_{\mathcal{H}},\Lambda^\nu)$.
\begin{proof}[Sketch of proof of Theorem \ref{theoremb}]
By Proposition \ref{taming}, $\mathcal{K}$ has a tame shrinking and by Proposition \ref{reduction} this tame atlas has a reduction. It is proven in \cite{mwtop} that the taming and reduction procedures are unique up to cobordism. By Proposition \ref{section}, one may construct a perturbation $\nu$ so that as above, the perturbed zero set $(|\widehat{\mathbf{Z}}^\nu|_{\mathcal{H}},\Lambda^\nu)$ is a compact weighted branched manifold. Moreover, there is a natural inclusion $\iota^\nu:|\widehat{\mathbf{Z}}^\nu|_{\mathcal{H}}\rightarrow|\mathcal{V}|\subset|\mathcal{K}|$. It is also proven in \cite{mwfund, mwiso} that the oriented cobordism class of $\iota^\nu$ is independent of choice of $\nu$. Finally, $[X]_{\mathcal{K}}^{vir}$ is defined by taking an appropriate inverse limit in rational \v{C}ech homology.
\end{proof}

\subsection{$C^1$ stratified smooth atlases}\label{c1ss}
In Section \ref{gwcharts}, we will consider the case of $X = \overline{\mathcal{M}}_{0,k}(A,J)$, the compact space of nodal $J$-holomorphic genus zero stable maps on a symplectic manifold in homology class $A$ with $k$ marked points modulo reparametrization. The atlas for $X$ will not be smooth (in the sense that all spaces and maps are smooth), but will instead be \textit{$C^1$ stratified smooth}. This section explains the basics of stratified spaces, their role in Gromov-Witten Kuranishi atlases, and shows that a $C^1$ stratified smooth structure is sufficient for our purposes.

\begin{definition}
A pair $(X,\mathcal{T})$ of a topological space $X$ and a finite partially ordered set $\mathcal{T}$ is called a \textbf{stratified space with strata} $(X_T)_{T\in\mathcal{T}}$ if the following conditions hold:
\begin{enumerate}[(i)]
\item The space $X$ is a disjoint union of the strata.
\item The closure $\overline{X_T}$ is contained in $\bigcup_{S\leq T}X_S$.
\end{enumerate}
\end{definition}
\begin{example}\label{stratex}
$\mathbb{R}^k\times\mathbb{C}^{\underline{n}}$ carries a stratification in the following way:
\begin{enumerate}
\item The set $\mathcal{T}$ is the power set of $\{1,\ldots,n\}$.
\item The stratum $(\mathbb{R}^k\times\mathbb{C}^{\underline{n}})_T$ is given by
\[(\mathbb{R}^k\times\mathbb{C}^{\underline{n}})_T = \{(x,\mathbf{a})\in\mathbb{R}^k\times\mathbb{C}^n ~|~ \mathbf{a} = (a_1,\ldots,a_n), a_i\neq0\iff i\in T\}.\]
\end{enumerate}
\end{example}
\begin{example}
The genus zero Deligne-Mumford space $\overline{\mathcal{M}}_{0,k}$ admits a stratification given by
\begin{equation}\label{dmstrat}
\overline{\mathcal{M}}_{0,k} = \bigsqcup_{0\leq p \leq k-3} \mathcal{M}_p
\end{equation}
where $\mathcal{M}_p$ is the moduli space of curves with exactly $p$ nodes. This is the stratification we will use.

The genus zero Gromov-Witten moduli space $\overline{\mathcal{M}}_{0,k}(A,J)$ also admits a stratification similar to (\ref{dmstrat}):

\[\overline{\mathcal{M}}_{0,k}(A,J) = \bigsqcup_{0\leq p \leq k-3} \mathcal{M}_p(A,J).\]

Deligne-Mumford space $\overline{\mathcal{M}}_{0,k}$ also admits another stratification given by
\[\overline{\mathcal{M}}_{0,k} = \bigsqcup_{\substack{k\textnormal{-labelled}\\ \textnormal{trees } T}} \mathcal{M}_T\]
where $\mathcal{M}_T$ is the moduli space of curves modelled on the labelled tree $T$.
\end{example}

\begin{definition}
A \textbf{stratified continuous} (resp. $\mathbf{C^1}$) map $f:(X,\mathcal{T})\rightarrow (Y,\mathcal{S})$ between stratified spaces (resp. $C^1$ manifolds) is a continuous (resp. $C^1$) map $f:X\rightarrow Y$ such that
\begin{enumerate}[(i)]
\item $f$ maps strata to strata and hence induces a map $f_{*}:\mathcal{T}\rightarrow\mathcal{S}$ such that $f(X_T)\subset Y_{f_{*}T}$.
\item The map $f_{*}$ preserves order, that is $T\leq S\Rightarrow f_{*}T \leq f_{*}S$.\footnote{This is a different definition than that in \cite{notes} where it is required that $f_{*}$ preserves \textit{strict} order. We use this definition so that relevant maps such as the evaluation map $ev:U_{I}\rightarrow M$ are stratified $C^1$.}
\end{enumerate}
\end{definition}

\begin{definition}
Let $f:U\rightarrow \mathbb{R}^k\times\mathbb{C}^{\underline{m}}$ be a stratified $C^1$ map defined on an open subset $U\subset \mathbb{R}^\ell\times \mathbb{C}^{\underline{n}}$. We call $f$ a \textbf{$\mathbf{C^1}$ stratified smooth ($\mathbf{C^1}$ SS) map} if it restricts to a smooth map $U_T=U\cap(\mathbb{R}^\ell\times\mathbb{C}^{\underline{n}})_T\rightarrow (\mathbb{R}^k\times\mathbb{C}^{\underline{m}})_{f_*T}$ on each stratum $T$.
\end{definition}

It is easy to see that the composition of two $C^1$ SS maps is $C^1$ SS. Hence one can define a \textbf{$\mathbf{C^1}$ SS manifold} as having local charts that are open subsets of $\mathbb{R}^k\times\mathbb{C}^{\underline{n}}$ as in Example \ref{stratex} and requiring that transition functions are $C^1$ SS maps. Also of use later will be the notion of a subset $Y$ being a \textbf{$C^1$ SS submanifold} of a $C^1$ SS manifold $X$. Here we say that $Y\subset X$ is a $C^1$ SS submanifold if $Y$ is a $C^1$ SS manifold with its $C^1$ SS structure induced from $X$. See Example \ref{localc1ss} below for examples of local models of $C^1$ SS submanifolds and later Example \ref{sphereexample} and Remark \ref{sssubrem} for more general instances in which $C^1$ SS submanifolds arise.

\begin{example}[Local $C^1$ SS submanifolds]\label{localc1ss}
Here we give two examples of $C^1$ SS submanifolds of the $C^1$ SS manifold $\mathbb{R}^k\times\mathbb{C}^{\underline{n}}$ with stratification described in Example \ref{stratex}. These will later serve as the local models of $C^1$ SS submanifolds as described in Remark \ref{sssubrem}.
\begin{enumerate}
\item\label{localc1ss1} $\{(x,\mathbf{a})~|~a_1=0\}\cong\mathbb{R}^k\times\mathbb{C}^{\underline{n-1}}\subset\mathbb{R}^k\times\mathbb{C}^{\underline{n}}$ is easily seen to be a $C^1$ SS submanifold.
\item\label{localc1ss2} $\{(x,\mathbf{a})~|~a_1=1\}\subset\mathbb{R}^k\times\mathbb{C}^{\underline{n}}$ is also seen to be a $C^1$ SS submanifold.
\end{enumerate}

\end{example}

One of the important properties of $C^1$ SS manifolds is that they admit a good transversality theory.
\begin{definition}\label{sstransverse}
Let $U$ be an open subset of $\mathbb{R}^k\times\mathbb{C}^{\underline{n}}$. Let $f:U\rightarrow \mathbb{R}^{\ell}$ be a $C^1$ SS map and $f(w)=0$ for $w\in U_T$. We say that $f$ is \textbf{SS transverse to $0$ at $w$} if the derivative $d_w^Tf:T_wU_T\rightarrow\mathbb{R}^{\ell}$ is surjective at $w$, where $d_w^T$ is the differential of the smooth restriction of $f$ to the stratum $U_T$. We say that $f$ is \textbf{SS transverse} to $0$ if it is SS transverse to $0$ at all $w\in f^{-1}(0)$.
\end{definition}

One can extend this in the obvious manner to maps $f$ with $C^1$ SS manifolds as domains and targets. There is also the notion of $f$ being SS transverse to a submanifold.

We notice that a map $f$ being SS transverse to $0$ is stronger than being transverse to $0$ in the usual sense. That is $d_w^Tf:T_wU_T\rightarrow\mathbb{R}^n$ being surjective implies that $d_wf:T_wU\rightarrow\mathbb{R}^n$ is surjective. Additionally, being SS transverse is an open condition because transversality for $C^1$ maps is open. SS transversality is also generic (see Lemma \ref{c1tran}). Compare this with the situation of weakly SS (not $C^1$) maps described in \cite{notes}.

The following lemma is obvious, but will be the basis of constructing a virtual fundamental class for $C^1$ SS weak Kuranishi atlases.
\begin{lemma}\label{sstransverselemma}
Let $X$ be a $C^1$ SS manifold and let $W$ be a smooth manifold with $Z\subset W$ a submanifold. If $f:X\rightarrow W$ is a $C^1$ SS map that is SS transverse to $Z$, then $f^{-1}(Z)$ is a $C^1$ SS submanifold of $X$.
\end{lemma}
\begin{example}\label{sphereexample}
The unit sphere $S^{2n-1}\subset\mathbb{C}^{\underline{n}}$ is a $C^1$ SS submanifold. This can either be seen directly by looking at charts for $S^{2n-1}$ or by noting that the map
\begin{align*}
f:\mathbb{C}^{\underline{n}}&\rightarrow\mathbb{R}\\
f(a_1,\ldots,a_n)&=|a_1|^2+\cdots+|a_n|^2
\end{align*}
is SS transverse to $1$ and hence $S^{2n-1}$ is cut out transversally on each stratum.
\end{example}
\begin{remark}\label{sssubrem}
Of course not every submanifold $Y$ of a $C^1$ SS manifold $X$ is a $C^1$ SS submanifold. That is, the stratification that $Y\subset X$ inherits from $X$ does not necessarily give $Y$ the structure of a $C^1$ SS manifold. However, there are two natural situations when this is the case. The first is the situation of Lemma \ref{sstransverselemma} when $Y$ is cut out by a SS transverse map. This is shown in Example \ref{sphereexample} and on the local level in Example \ref{localc1ss}.\ref{localc1ss2}. The second situation is when there are charts $X \supset U \rightarrow \mathbb{R}^k\times \mathbb{C}^{\underline{n}}$ such that $U\cap Y\subset X$ is given by $\mathbb{R}^k\times\mathbb{C}^{\underline{\ell}}\times\{0\}$. On the local level, this is Example \ref{localc1ss}.\ref{localc1ss1}.\hfill$\Diamond$
\end{remark}

One can also define the notion of a \textbf{$\mathbf{C^1}$ SS Kuranishi atlas} on a stratified space $X$ by requiring all conditions to hold in the $C^1$ SS category. In particular, the footprint map should be stratified continuous and the domain $\widetilde{U}_{IJ}\subset s_J^{-1}(E_I)\subset U_J$ should be a $C^1$ SS submanifold. The latter fact follows from the index condition referred to in Definition \ref{transitiondef}. Clearly every smooth atlas is also $C^1$ SS. We are interested in $C^1$ SS atlases because of Theorem \ref{admitthm}, which states that Gromov-Witten moduli spaces admit $C^1$ SS atlases; this atlas will be described in Section \ref{gwcharts}. We now explain the proof of Theorem \ref{c1thm}, which states that $C^1$ SS atlases have enough structure to define a virtual fundamental class. This proof relies on results from Section \ref{sectionsc1}.

\begin{proof}[Proof of Theorem \ref{c1thm}]
Recall from Section \ref{atlases} that the construction of the virtual fundamental cycle takes place in three steps: taming, reduction, and construction of a perturbation section. The taming and reduction constructions carry over with no changes. These procedures are topological in nature and can be done in much in the more general setting of topological atlases, of which $C^1$ SS atlases are examples. The issue is to construct an appropriate perturbation section. Proposition \ref{wbm} and the discussion beforehand explains that it suffices to construct the perturbation section in the absence of isotropy.  It would suffice to construct a $C^1$ transverse perturbation because this would make the perturbed zero set a $C^1$ manifold, which carries a fundamental class. However, we prove the even stronger existence of a $C^1$ SS transverse perturbation. The precise statement of this existence is contained in Proposition \ref{c1section} and its proof is contained in Section \ref{sectionsc1}. By Lemma \ref{sstransverselemma}, the existence of a $C^1$ SS transverse perturbation implies that the perturbed zero set is a $C^1$ SS manifold and hence allows for the construction of the virtual fundamental cycle. This completes the proof of Theorem \ref{c1thm}.
\end{proof}

\subsection{Gromov-Witten charts}\label{gwcharts}
In this section we will give an outline of the construction of the charts for a $C^1$ SS Kuranishi atlas on Gromov-Witten moduli spaces as in Theorem \ref{admitthm}. For more details and proof that these constitute a $C^1$ SS atlas, the reader should refer to \cite{gluingme}. This description of the charts will be used in Sections \ref{constraints} and \ref{gwaxioms}.

Let $(M^{2n},\omega)$ be a compact symplectic manifold and let $J$ be a tame almost complex structure. Let $X = \overline{\mathcal{M}}_{0,k}(A,J)$ be the compact space of nodal $J$-holomorphic genus zero stable maps in homology class $A$ with $k$ marked points modulo reparametrization. Let $d=2n+2c_1(A)+2k-6$. We will think of elements of $X$ as equivalence classes $[\Sigma,\mathbf{z},f]$ where $\Sigma$ is a genus zero nodal Riemann surface, $\mathbf{z}$ are $k$ disjoint marked points, and $f:\Sigma\rightarrow M$ is $J$-holomorphic map in homology class $A$. We will describe charts for an atlas on $X$ in a coordinate free manner. While it is relatively easy to describe the charts in this way, it is difficult to prove anything. A description of charts in explicit coordinates can be found in \cite{gluingme} and \cite[$\S 4$]{notes}. Let $[\Sigma_0,\mathbf{z}_0,f_0]\in X$, we will describe a chart around this element.\\

\noindent \textbf{Isotropy group}: The isotropy group $\Gamma$ is the isotropy of the element $[\Sigma_0,\mathbf{z}_0,f_0]$.\\

\noindent \textbf{Slicing manifold}: Next, we make an auxiliary choice of a \textbf{slicing manifold} $Q$ that is transverse to im$f_0$, disjoint from $f_0(\mathbf{z}_0)$, orientable, and so that the $k$ marked points $\mathbf{z}_0$ together with the $L$ marked points $\mathbf{w}_0:=f_0^{-1}(Q)$ stabilize the domain of $f_0$.\\

\noindent\textbf{Obstruction space}: Let
\begin{equation}\label{delta}
\Delta\subset\overline{\mathcal{M}}_{0,k+L}
\end{equation}
be a small neighborhood of $[\Sigma_0,\mathbf{w}_0,\mathbf{z}_0]\in\overline{\mathcal{M}}_{0,k+L}$ and $\mathcal{C}|_{\Delta}$ denote the universal curve over $\Delta$.

To define the obstruction space, we first choose a map $\lambda:E_0\rightarrow C^\infty(Hom_J^{0,1}(\mathcal{C}|_{\Delta}\times M))$. The space $E_0$ and the map $\lambda$ are chosen to ensure that the domains are cut out transversally near $f_0$ (see (\ref{domainhateqn}) and the discussion afterwards). Then the obstruction space $E$ is defined to be $E:=\Pi_{\gamma\in\Gamma}E_0$, the product of $|\Gamma|$ copies of $E_0$ with elements $\vec{e}:=(e^\gamma)_{\gamma\in\Gamma}$. The isotropy group $\Gamma$ acts by permutation
\begin{equation}\label{actonobs}
(\alpha\cdot\vec{e})^{\gamma} = e^{\alpha\gamma}.
\end{equation}
The map $\lambda$ is then extended to be a $\Gamma$-equivariant map on $E$.\\

\noindent \textbf{Defining $\widehat{U}$}: Before defining the domain $U$, we define $\widehat{U}$. For $(0,[\Sigma_0,\mathbf{w}_0,\mathbf{z}_0,f_0])\in X$, $\widehat{U}$ is defined to be a neighborhood of $(0,[\Sigma_0,\mathbf{w}_0,\mathbf{z}_0,f_0])$ in the following space
\begin{equation}\label{domainhateqn}
\widehat{U}\subset\left\{(\vec{e},[\Sigma,\mathbf{w},\mathbf{z},f])~|~\vec{e}\in E,~ \overline{\partial}_Jf = \lambda(\vec{e})|_{\textnormal{graph}f}\right\}.
\end{equation}
Note that $\bar{\partial}_J$ gives a section of the bundle $C^\infty\big(Hom_J^{0,1}(\mathcal{C}|_{\Delta}\times M)\big)$. Standard Fredholm theory implies that while its linearization $d_f(\bar{\partial}_J)$ may not be surjective, it has a finite dimensional cokernel so that one can choose $\lambda(E_0)$ to surject onto it. The solution set $\widehat{U}$ in (\ref{domainhateqn}) is then a $C^1$ SS manifold. See \cite{notes} for more details on how $E$ and $\lambda$ are chosen.

The space $\widehat{U}$ inherits a smooth structure from (reparametrized) gluing maps. Showing that this can be done is a major subject of \cite{gluingme}.

\begin{remark}\label{maps}
The space $\widehat{U}$ (and later the domain $U$) carry important maps. The first is a forgetful map
\begin{equation}\label{forgetful1}
\pi_0:\widehat{U}\rightarrow \overline{\mathcal{M}}_{0,k}^{new}.
\end{equation}
that takes an element $(\vec{e},[\Sigma,\mathbf{w},\mathbf{z},f])$ to the stabilization of the underlying domain of $f$. It is shown in \cite{gluingme} that $\overline{\mathcal{M}}_{0,k}$ admits a $C^1$ SS structure, denoted $\overline{\mathcal{M}}_{0,k}^{new}$ such that $\pi_0$ is $C^1$ SS (see Proposition \ref{dm1}). We will sometimes forgo the ``$new$" for simplicity as this is the only smooth structure we will use.

The space $\widehat{U}$ also carries an evaluation map
\begin{align*}
ev_k:\widehat{U} &\rightarrow M^k\\
(\vec{e},[\Sigma,\mathbf{w},\mathbf{z},f]) &\mapsto (f(z_1),\ldots,f(z_k)).
\end{align*}
The smooth structure on $\widehat{U}$ is defined is such a way that $ev_k$ is $C^1$ SS.\hfill$\Diamond$
\end{remark}

\noindent \textbf{Isotropy group action}: The group $\Gamma$ is the stabilizer of $[\Sigma_0,\mathbf{z}_0,f_0]$. Each $\gamma\in\Gamma$ is uniquely determined by how it permutes the extra marked points $\mathbf{w}_0$. That is, we consider $\Gamma$ to be a subgroup of the symmetric group $S_L$ and $\Gamma$ acts by
\[\mathbf{w}_0\mapsto \gamma\cdot\mathbf{w}_0:=(w^{\gamma(\ell)})_{1\leq\ell\leq L}.\]
There is also an associated action on the nodes $\mathbf{n}_0$.

We can extend this action to an action on $\Delta$ by
\[\delta=[\mathbf{n},\mathbf{w},\mathbf{z}]\mapsto\gamma^{*}(\delta)=[\gamma\cdot\mathbf{n},\gamma\cdot\mathbf{w},\mathbf{z}].\]
Finally, $\Gamma$ (partially) acts on $\widehat{U}$ by
\begin{equation}\label{coordfreeiso}
(\vec{e},[\mathbf{n},\mathbf{w},\mathbf{z},f])\mapsto(\gamma\cdot\vec{e},[\gamma\cdot\mathbf{n},\gamma\cdot\mathbf{w},\mathbf{z},f])
\end{equation}
where the action $\gamma\cdot\vec{e}$ is given by (\ref{actonobs}). It is proved in \cite{notes} that (\ref{coordfreeiso}) preserves solutions to (\ref{domainhateqn}).\\

\noindent \textbf{Domain}: The domain $U$ is then obtained by making $\widehat{U}$ $\Gamma$-invariant and then cutting down $\widehat{U}$ via the slicing manifold $Q$. By this we mean we may assume $\widehat{U}$ is $\Gamma$-invariant by taking an intersection of the finite orbit of $\widehat{U}$ under $\Gamma$. Then we cut down via the slicing manifold, meaning we make $U$ a $\Gamma$-invariant open neighborhood of $(0,[\Sigma_0,\mathbf{w}_0,\mathbf{z}_0,f_0])$ in the following space:
\begin{equation}\label{slice}
U\subset \left\{(\vec{e},[\Sigma,\mathbf{w},\mathbf{z},f])\in\widehat{U}~|~f(\mathbf{w})\subset Q\right\}.
\end{equation}
The subset $\widehat{U}$ will be of use because the analysis requires working in $\widehat{U}$. (See \cite[Remark 4.1.2]{notes} for more details on the necessity of working with $\widehat{U}$.)

The \textbf{section} and \textbf{footprint} maps are given by
\[s\left(\vec{e},[\Sigma,\mathbf{w},\mathbf{z},f]\right) = \vec{e},\qquad \psi\left(\vec{0},[\Sigma,\mathbf{w},\mathbf{z},f]\right) = [\Sigma,\mathbf{z},f].\]

The construction of the sum charts is very similar. For $I\subset\mathcal{I}_{\mathcal{K}}$, denote $\underline{\vec{e}}:=(\vec{e_i})_{i\in I}, \underline{\mathbf{w}}:=(\mathbf{w}_i)_{i\in I}$. Then $U_I$ is chosen to be a $\Gamma_I$-invariant open set
\[U_I\subset \left\{(\underline{\vec{e}},[\Sigma,\underline{\mathbf{w}},\mathbf{z},f])~|~\underline{\vec{e}}\in E_I, f(\mathbf{w}_i)\subset Q_i, \overline{\partial}_Jf = \lambda(\underline{\vec{e}})|_{\textnormal{graph}f}\right\}\]
such that it has footprint $F_I=\bigcap_{i\in I}F_i$. As before, $U_I$ is obtained from a $\widehat{U}_I$ by cutting down via the slicing manifolds $\{Q_i\}_{i\in I}$.

In this coordinate free language, the coordinate changes $\mathbf{K}_I\rightarrow\mathbf{K}_J$ are given by choosing an appropriate domain $\widetilde{U}_{IJ}\subset U_J$ and then simply forgetting the components $(\mathbf{w_i})_{i\in(J\setminus I)}$ and the $(e_i)_{i\in (J\setminus I)}$ (which are 0 because $\widetilde{U}_{IJ}\subset s_J^{-1}(E_I)$).

\section{Constructing Kuranishi atlases with constraints}\label{constraints}
In Section \ref{gwatlases} we outlined the construction of the virtual fundamental class for the space $X=\overline{\mathcal{M}}_{0,k}(A,J)$ using the results of \cite{mwtop, mwfund, mwiso, gluingme} and Theorem \ref{c1thm}. There are two common variants on the space $X$: considering subsets of $X$ where the domain and/or image of $[\Sigma,\mathbf{z},f]\in X$ are constrained by a homology class. Sections \ref{homconstruction} and \ref{domainconstruction} describe these invariants, the construction of Kuranishi atlases for these invariants, and prove results based on the notion of a \textit{transverse subatlas} introduced in Section \ref{tsubatlases}.

\subsection{Homological constraints on the image}\label{homconstruction}
In this section, we will describe Gromov-Witten invariants with homological constraints in $M$. There are two ways one could define such invariants. The main result of this section will be that these two definitions agree. Let $c_1,\ldots,c_k\in H_{*}(M)$. We wish to consider elements of $X$ that constrain the image of the evaluation map $ev_k([\Sigma,\mathbf{z},f]) := (f(z_1),\ldots,f(z_k))\in M^k$. To do this, we can build a Kuranishi atlas $\mathcal{K}$ for $X$ by the procedure described in Section \ref{gwcharts} and then constrain homologically. That is, one can form $[X]_\mathcal{K}^{vir}\in \check{H}_{d}(X;\mathbb{Q})$, where $d:=$ ind$(A)=2n+2c_1(A)+2k-6$, (using Theorems \ref{admitthm} and \ref{c1thm}) and then push forward this class using the evaluation map $ev_k:X\rightarrow M^k$ to $(ev_k)_{*}([X]_\mathcal{K}^{vir})\in H_{*}(M^k)$. Then define the Gromov-Witten invariant
\[GW_{A,k}^M(c_1,\ldots, c_k) = (ev_k)_{*}([X]_\mathcal{K}^{vir})\cdot (c_1\times\cdots\times c_k) \in \mathbb{Q}\]
using the intersection product in $M^k$ (so this is $0$ unless $\sum_{i=1}^k(2n-\deg c_i) = d$).

Alternatively, one could directly constrain the space $X$. It is enough to consider homology classes that are represented by submanifolds because we will deal only with rational homology; let $Z_c\subset M^k$ be a closed submanifold representing the homology class
\begin{equation}\label{cdef}
c=c_1\times\cdots\times c_k\in H_{\dim{c}}(M^k).
\end{equation}
Consider the subset of $X$
\[X \supset X_c := \overline{\mathcal{M}}_{0,k}(A,J;c) = \{[\Sigma,\textbf{z},f]\in\overline{\mathcal{M}}_{0,k}(A,J)~|~ev_k(f)\in Z_c\}.\]
If $d:=$ ind$(A)$ is the formal dimension $2n+2c_1(A)+2k-6$ of $\overline{\mathcal{M}}_{0,k}(M,J,A)$, then the subset $X_c$ has formal dimension $d-$codim $c = 0$. One could then build an atlas $\mathcal{K}_c$ for $X_c$ directly. The construction of such charts is similar to the construction of the charts for $X$; in fact, roughly speaking, charts for $X_c$ are obtained from charts for $X$ by imposing homological constraints.

We now give an overview of the construction of $\mathcal{K}_c$ (for a more detailed treatment, see \cite{notes}); we will use the description and notion for charts for $X$ from Section \ref{gwcharts}. As with the charts for $X$, given a center point $[\Sigma_0,\mathbf{z}_0,f_0]\in X_c$ we first define
\begin{equation}\label{uchat}
\widehat{U}_c\subset\left\{(\vec{e},[\Sigma,\mathbf{w},\mathbf{z},f])\in\widehat{U} ~|~ ev_k(f) \in Z_c\right\}
\end{equation}
to be an open neighborhood of $(\vec{0},\mathbf{w}_0,\mathbf{z}_0,f_0)$. To ensure that this is a $C^1$ SS manifold of dimension dim$(\widehat{U})-$codim$(c)$, we need the evaluation map $ev_k:\widehat{U}\rightarrow M^k$ to be SS transverse to $Z_c$. The obstruction space $E$ is chosen so that this is true for sufficiently small $\widehat{U}$. Then, as in the construction of the charts for $X$, we make $\widehat{U}_c$ $\Gamma$-invariant by taking an intersection of the finite orbit of $\widehat{U}_c$ under $\Gamma$. Then the domain $U_c$ of a basic chart for $X_c$ is obtained from $\widehat{U}_c$ by cutting down using the slicing manifold $Q$ as in (\ref{slice}). The rest of the construction goes through as before.

In this way we form a $0$-dimensional atlas $\mathcal{K}_c$ for $X_c$. The virtual class $[X_c]_{\mathcal{K}_c}^{vir}$ is then a rational number. The next theorem says that this is the same as the Gromov-Witten invariant $GW^M_{A,k}(c_1,\ldots, c_k)$.

\begin{theorem}\label{homologyagree}
The two approaches of defining Gromov-Witten invariants described above agree. That is
\[GW^M_{A,k}(c_1,\ldots, c_k) := (ev_k)_{*}([X]_\mathcal{K}^{vir})\cdot (c_1\times\cdots\times c_k) = [X_c]_{\mathcal{K}_c}^{vir} \in \mathbb{Q}.\]
\end{theorem}

In order to prove this theorem, we introduce the notion of a \emph{transverse subatlas} (of which $\mathcal{K}_c$ will be an example, see Lemma \ref{xcit}) and prove basic results. Theorem \ref{homologyagree} will then easily follow from this machinery and is proved at the end of Section \ref{tsubatlases}.

\subsection{Transverse subatlases}\label{tsubatlases}
This section defines and proves basic properties of transverse subatlases. One should think of a transverse subatlas as the appropriate notion for an atlas cut out transversally by an evaluation map; the atlas $\mathcal{K}_c$ of Section \ref{homconstruction} (see (\ref{uchat})) is the prototypical example. For simplicity we will begin by working with smooth atlases before discussing the necessary changes for $C^1$ SS atlases. The main result is the following.

\begin{proposition}\label{tprop}
Let $\mathcal{K}_c$ be a transverse subatlas of $\mathcal{K}$. Then the procedures of taming, reduction, and construction of the perturbation section for each atlas can be done in a such a way that the perturbed zero set $|\mathbf{Z}^{\nu_c}|_{\mathcal{H}}$ is a compact weighted branched submanifold of the perturbed zero set $|\mathbf{Z}^{\nu}|_{\mathcal{H}}$ with each branch cut out transversally by the evaluation map $f$. That is on each branch the evaluation map $f:|\mathbf{Z}^{\nu}|_{\mathcal{H}}\rightarrow N$ is transverse to the submanifold $C\subset N$ in Definition \ref{tdefinition} and
\[|\mathbf{Z}^{\nu_c}|_{\mathcal{H}}=f^{-1}\left(C\right).\]
\end{proposition}

To prove Proposition \ref{tprop} we must establish relative versions of the taming, reduction, and perturbation section constructions. Precise statements of these results can be found in Lemmas \ref{ttaming}, \ref{treduction}, \ref{tsection}. Lemmas \ref{ttaming} and \ref{treduction} will be proved in this section, while the proof of Lemma \ref{tsection} is deferred to Section \ref{tsections}.

\begin{definition}\label{subatlasdeff}
Let $X_c$ be a closed subset of a compact metric space $X$. We say that an atlas $\displaystyle \mathcal{K}_c=(\mathbf{K}_I^c,\widehat{\Phi}_{IJ}^c)_{I,J\in\mathcal{I}_{\mathcal{K}_c},I\subsetneq J}$ for $X_c$ is a \textbf{subatlas of codimension $d_c$} of an atlas $\displaystyle \mathcal{K}=(\mathbf{K}_I,\widehat{\Phi}_{IJ})_{I,J\in\mathcal{I}_{\mathcal{K}},I\subsetneq J}$ for $X$ if the following conditions hold:
\begin{enumerate}[(i)]
\item There is an injection of the indexing set $\mathcal{I}_{\mathcal{K}_c}$ into the indexing set $\mathcal{I}_\mathcal{K}$. For simplicity we will view $\mathcal{I}_{\mathcal{K}_c}$ as a subset $\mathcal{I}_{\mathcal{K}_c}\subset\mathcal{I}_\mathcal{K}$.
\item The isotopy group $\Gamma_I^c$ of $\mathbf{K}_I^c$ is the same as the isotropy group $\Gamma_I$ of $\mathbf{K}_I$.
\item For $I\in\mathcal{I}_{\mathcal{K}_c}$, the domain $U_I^c$ of the chart $\mathcal{K}_I^c$ is a closed $\Gamma_I$-invariant smooth submanifold of codimension $d_c$ of the domain $U_I$ of $\mathbf{K}_I$.
\item The obstruction space $E_I^c$  of $\mathbf{K}_I^c$ is the same as the obstruction space $E_I$ of $\mathbf{K}_I$.
\item The section $s_I:U_I\rightarrow E_I$ agrees with the section $s_I^c$ on $U_I^c\subset U_I$, that is $s_I^c = s_I|_{U_I^c}$.
\item The footprint map $\psi_I:s_I^{-1}(0)\rightarrow X$ agrees with the footprint map $\psi_I^c:(s_I^c)^{-1}(0)\rightarrow X_c$ on $(s_I^c)^{-1}(0)\subset s_I^{-1}(0)$. Moreover, $F_I\cap X_c = F_I^c$.
\item\label{tcoordinatechange} The coordinate changes $\widehat{\Phi}_{IJ}^c$ are given by the restriction of $\widehat{\Phi}_{IJ}$ to $U_{IJ}^c \subset U_{IJ}$. More precisely, $\widetilde{U}_{IJ}^c$ is a closed smooth submanifold of $\widetilde{U}_{IJ}$ and
    \[\rho_{IJ}^c = \rho_{IJ}|_{\widetilde{U}_{IJ}^c}.\]
\end{enumerate}
\end{definition}
\begin{definition}\label{tdefinition}
Let $X_c$ be a closed subset of a compact metric space $X$. Let $\mathcal{K}_c$ be an atlas for $X_c$ that is a subatlas of codimension $d_c$ of an atlas $\mathcal{K}$ for $X$. The subatlas $\mathcal{K}_c$ is called a \textbf{transverse subatlas} if the following conditions hold:
\begin{enumerate}[(i)]
\item The charts $\mathbf{K}_I$ of $\mathcal{K}$ carry a smooth ``evaluation map" $f_I:U_I\rightarrow N$ to a smooth manifold N.
\item\label{evcoordinatechange} The evaluation map $f_I$ respects the coordinate changes $\widehat{\Phi}_{IJ}$ and the group action of $\Gamma_I$ in the sense that $f_I\circ\rho_{IJ} = f_J$ on $\widetilde{U}_{IJ}$ and $f_I\circ\gamma = f_I$ for all $\gamma\in\Gamma_I$.
\item\label{conditioniii} There is a closed smooth submanifold $C\subset N$ of codimension $d_c$ such that $f_I\pitchfork C$ and the domain $U_I^c$ of the chart $\mathcal{K}_I^c$ is the closed smooth submanifold $U_I^c = f_I^{-1}(C)$.
\item\label{conditioniv} There is an $\varepsilon>0$ such that if $I\notin\mathcal{I}_{\mathcal{K}_c}$, then $f_I(U_I)\subset N$ lies outside an $\varepsilon$-neighborhood of $C\subset N$.
\end{enumerate}
\end{definition}
\begin{remark}
\begin{enumerate}[(a)]\label{tremark}
\item Condition $(\ref{evcoordinatechange})$ of Definition \ref{tdefinition} implies that $X$ carries an evaluation map $f_X:X\rightarrow N$ defined by the inverses of the footprint maps. Moreover, if $\mathcal{K}$ is a Kuranishi atlas so that the topological realization $|\mathcal{K}|$ is defined (see Definition \ref{catdef} and the discussion afterwards for more about $|\mathcal{K}|$), then $|\mathcal{K}|$ carries an evaluation map $f_{|\mathcal{K}|}:|\mathcal{K}|\rightarrow N$. The maps fit into the following commutative diagram
     \[\xymatrix{
X \ar[r]^{\iota_{\mathcal{K}}} \ar[rd]_{f_X} &
|\mathcal{K}| \ar[d]^{f_{|\mathcal{K}|}}\\
&N}\]
where $\iota_{\mathcal{K}}:X\hookrightarrow|\mathcal{K}|$ is the natural inclusion. For this reason, we will sometimes abuse notation and refer to the evaluation map $f$ without specifying its domain. Similarly, the rest of Definition \ref{tdefinition} implies that analogous statements hold true for the atlas $\mathcal{K}_c$ with the target of the evaluation map being $C$.
\item In the terminology of \cite[$\S 4.2$]{notes}, a transverse subatlas $\mathcal{K}_c$ of $\mathcal{K}$ induces an \textbf{embedding} of $\mathcal{K}_c$ to $\mathcal{K}$. A transverse subatlas has the additional properties that the obstruction bundles and isotropy groups agree, $\mathcal{K}$ carries an evaluation map $f$ that defines $\mathcal{K}_c$ (by Condition $(\ref{conditioniii}))$, and
    \[I\in\mathcal{I}_\mathcal{K},\quad F_I\cap X_c \neq \emptyset \iff I\in \mathcal{I}_{\mathcal{K}_c}\]
    (by Condition $(\ref{conditioniv}))$. Conditions $(\ref{conditioniii})$ and $(\ref{conditioniv})$ are used to separate the atlas $\mathcal{K}_c$ from the rest of $\mathcal{K}$ and are crucial to the geometric arguments we will make later in this section.

\item An elementary case when a transverse subatlas appears is when extending the atlas for a closed submanifold of a compact manifold to an atlas on the entire manifold. Here we think of building an atlas for a manifold with one chart and a trivial obstruction space. More generally, the same is true for orbifolds and suborbifolds, where an orbifold is the realization of a Kuranishi atlas with trivial obstruction space (see \cite[$\S 5$]{notes}).
\end{enumerate}
\end{remark}

We now show that the operations on atlases described in Section \ref{atlases} that McDuff and Wehrheim use to construct the virtual fundamental cycle extend to the relative case.

\begin{lemma}\label{ttaming}
Let $\mathcal{K}_c$ be a transverse subatlas of an atlas $\mathcal{K}$. Then there is a tame shrinking $\mathcal{K}'$ of $\mathcal{K}$ such that the domains
\begin{equation}\label{ttamingint}
(U_I^c)' := U_I'\cap U_I^c,\qquad (U_{IJ}^c)' := U_{IJ}'\cap U_I^c
\end{equation}
define a tame shrinking $\mathcal{K}_c'$ of $\mathcal{K}_c$. In particular, $\mathcal{K}_c'$ is a transverse subatlas of $\mathcal{K}'$ and each domain $(U_I^c)'$ is cutout transversally from the domain $U_I'$ by the same evaluation map that defines $\mathcal{K}_c$ as a transverse subatlas of $\mathcal{K}$.
\end{lemma}
Before proceeding to the proof of Lemma \ref{ttaming}, it is worth noting that the proof of this lemma, as well as the proofs of Lemmas \ref{treduction} and \ref{tsection} are inspired by the analogous statements concerning cobordisms as stated in \cite{mwtop}. A transverse subatlas is similar to the end of a cobordism, with the domains being submanifolds instead of having a collared structure. In contrast to cobordisms, the section and coordinate changes of a transverse subatlas do not have a prescribed structure in the normal direction.

\begin{proof}
We first deal with shrinking on the level of footprints. We will prove that there is a suitable shrinking $\{F_i'\}$ of $\{F_i\}$ such that $\{F_i^{c,\prime} := F_i'\cap X_c\}$ defines a shrinking of $\{F_i^c\}$ where $\{F_i\},\{F_i^c\}$ are the footprint covers coming from $\mathcal{K},\mathcal{K}_c$ respectively. First use Proposition \ref{taming} to choose a tame shrinking $\mathcal{K}_c''$ of $\mathcal{K}_c$. Using the notation of Definition \ref{intermediatedef}, let $\underline{P}_I\sqsubset \underline{U}_I$ be an open subset such that $\underline{P}_I\cap \underline{U}_I^c = \underline{U}_I^{c,\prime\prime}$. For $I\in\mathcal{I}_{\mathcal{K}_c}$, define
\begin{equation}\label{Hi}
H_I := \underline{\psi}_I\left((\underline{s}_I^{-1}(0)\cap \underline{P}_I\right).
\end{equation}
Note that $H_I$ is a precompact open subset of $F_I$ such that $H_I\cap X_c = F_I^{c,\prime\prime}$. Furthermore, because $X_c$ is compact, there is some $\delta > 0$ such that a $2\delta$-neighborhood of $X_c$ is contained in $\cup_IH_I$.

Let $\{G_i\}$ be a shrinking of $\{F_i\}$. We claim that
\[F_i' := \left(G_i\setminus\overline{B_{\delta}(X_c)}\right)\cup H_i\]
is the desired shrinking. Here we use $B_\delta$ to denote a $\delta$-metric neighborhood in $X$ and are using the convention that $H_i=\emptyset$ if $\{i\}\notin\mathcal{I}_{\mathcal{K}_c}$. It is clear the $F_i'$ is open in $F_i$ and is precompact in $F_i$ because $H_i$ is precompact. The collection $\{F_i'\}$ still covers $X$ because $\{G_i\}$ covers $X$ and $\{H_i\}$ covers a $2\delta$ neighborhood of $X_c$. Moreover,
\[F_i^{c,\prime} := F_i'\cap X_c = H_i\cap X_c = F_I^{c,\prime\prime}\]
so $\{F_i^{c,\prime}\}$ is a shrinking of $\{F_i\}$. Finally, requirement $(\ref{shrinkingeq})$ in the definition of a shrinking is satisfied because it is satisfied for $\{F_i^{c,\prime}\}$.

Therefore we have a shrinking $\{F_i'\}$ of $\{F_i\}$ such that intersecting with $X_c$ gives a shrinking $\{F_i^{\prime,c}\}$ of $\{F_i^c\}$. Now apply Proposition \ref{taming} to construct a tame shrinking $\mathcal{K}'\sqsubset \mathcal{K}$ with footprints $\{F_i'\}$. Then use (\ref{ttamingint}) to define the domains of $\mathcal{K}_c'$. The footprint cover induced from $\mathcal{K}_c'$ is then $\{F_i^{\prime,c}\}$ which is a shrinking of $\{F_i^c\}$. The other conditions of Definition \ref{kshrinking} are easily seen to be satisfied, so $\mathcal{K}_c'$ is a shrinking of a $\mathcal{K}_c$. The tameness of $\mathcal{K}'$ implies the tameness of $\mathcal{K}_c'$.
\end{proof}

Next we show that process of reduction can be done compatibly for a transverse subatlas.

\begin{lemma}\label{treduction}
Let $\mathcal{K}_c$ be a tame transverse subatlas of a tame atlas $\mathcal{K}$. Then there is a reduction $\mathcal{V}_c$ of $\mathcal{K}_c$ and a reduction $\mathcal{V}$ of $\mathcal{K}$ such that the domains $V_I, V_I^c$ of $\mathcal{V},\mathcal{V}_c$ respectively satisfy
\[V_I\cap U_I^c = V_I^c,\]
so each domain of $\mathcal{V}_c$ is cutout transversally from a domain of $\mathcal{V}$ by the same evaluation map that defines $\mathcal{K}_c$ as a transverse subatlas of $\mathcal{K}$. We call such reductions $\mathcal{V},\mathcal{V}_c$ a \textbf{compatible transverse reduction}.
\end{lemma}
We will prove Lemma \ref{treduction} in steps. The following two preliminary lemmas will  be useful. We will use the notation of Definition \ref{intermediatedef}.
\begin{lemma}\cite[Lemma 2.2.7]{mwiso}\label{514}
Let $\mathbf{K}$ be a Kuranishi chart. Then for any open subset $F'\subset F$, there is a restriction $\mathbf{K}'$ to $F'$ with domain $\underline{U}'$ such that $U' := \pi^{-1}(\underline{U}')$ satisfies $\overline{U'}\cap s^{-1}(0)=\psi^{-1}(\overline{F'})$. Moreover, if $F'\sqsubset F$, then $\underline{U}'$ can be chosen to be precompact in $\underline{U}$.
\end{lemma}

Our first step in proving Lemma \ref{treduction} is dealing with reductions on the level of footprints. A precise statement of this is made below as Lemma \ref{footprintreduction}. We begin with a preliminary definition and lemma regarding reductions.

\begin{definition}\label{coverreduction}
Let $X$ be a compact Hausdorff space and $X=\cup_{i=1,\ldots,N}F_i$ be a finite open cover. A collection $\{Z_I\}_{I\subset\{1,\ldots,N\}}$ is called a \textbf{cover reduction} if the $Z_I\subset X$ satisfy
\begin{enumerate}[(i)]
\item $Z_I\sqsubset F_I:=\bigcap_{i\in I}F_i$ is a (possibly empty) open subset.
\item If $\overline{Z_I}\cap\overline{Z_J}\neq\emptyset,$ then $I\subset J$ or $J\subset I$.
\item $X=\bigcup_{I}Z_I$.
\end{enumerate}
\end{definition}
Note that if $\mathcal{K}$ is a Kuranishi atlas on $X$ with footprints $\{F_i\}$ and $\mathcal{V}$ is a reduction of $\mathcal{K}$ in the sense of Definition \ref{reductiondef}, then the reduced footprints $\left\{Z_I := \psi_I\left(V_I\cap s_I^{-1}(0)\right)\right\}_{I\subset\{1,\ldots,N\}}$ are a cover reduction of $\{F_i\}$. The next lemma says cover reductions always exist.
\begin{lemma}\cite[Lemma 5.3.1]{mwtop}\label{coverreductionlem}
Let $X$ be a compact Hausdorff space and $X=\cup_{i=1,\ldots,N}F_i$ be a finite open cover. Then there exists a cover reduction $\{Z_I\}_{I\subset\{1,\ldots,N\}}$. Moreover, $\{Z_I\}$ can be chosen to be of the form
\begin{equation}\label{712new}
Z_I := \left(\bigcap_{i\in I}G_i^{|I|}\right)\setminus\bigcup_{j\notin I}\overline{F_ j^{|I|}}
\end{equation}
where $F_i^{j},G_i^{j}$ are open sets of $F_i$ such that
\begin{equation}\label{711new}
F_i^{0}\sqsubset G_i^{1}\sqsubset F_i^{1}\sqsubset G_i^{2}\sqsubset\cdots\sqsubset F_i^{N}\sqsubset F_i
\end{equation}
and $\{F^{0}_i\}$ still covers $X$.
\end{lemma}

The following lemma is the first step in the proof of Lemma \ref{treduction} and says that compatible transverse reductions can be done on the level of footprints.
\begin{lemma}\label{footprintreduction}
Let $\mathcal{K}_c$ be a tame transverse subatlas of a tame atlas $\mathcal{K}$. Let $\{F_i^c\}, \{F_i\}$ be the footprint covers of $\mathcal{K}_c, \mathcal{K}$. Then there is a cover reduction $\{Z_I^c\}$ of $\{F_i^c\}$ and a cover reduction $\{Z_I\}$ of $\{F_i\}$ such that
\[Z_I^c:=Z_I\cap X_c.\]
\end{lemma}

\begin{proof}
By Lemma \ref{coverreductionlem}, $\{F_i^c\}$ has a cover reduction $\{Z_I^c\}$ of the form
\begin{equation}\label{712}
Z_I^c := \left(\bigcap_{i\in I}G_i^{c, |I|}\right)\setminus\bigcup_{j\notin I}\overline{F_ j^{c,|I|}}
\end{equation}
where $F_i^{c,j},G_i^{c,j}$ are open subsets of $F_i^c$ such that
\begin{equation}\label{711}
F_i^{c,0}\sqsubset G_i^{c,1}\sqsubset F_i^{c,1}\sqsubset G_i^{c,2}\sqsubset\cdots\sqsubset F_i^{c,N^c}\sqsubset F_i^c
\end{equation}
and $\{F^{c,0}_i\}$ still covers $X_c$. Use Lemma \ref{514} to construct open domains
\[\underline{U}^{c,0}_{i,F}\sqsubset \underline{U}^{c,1}_{i,G}\sqsubset\cdots\sqsubset \underline{U}^{c,N^c}_{i,F}\sqsubset \underline{U}_i^c\]
with footprints (\ref{711}). We then wish to construct open sets
\[\underline{U}^{0}_{i,F}\sqsubset \underline{U}^{1}_{i,G}\sqsubset\cdots\sqsubset \underline{U}^{N^c}_{i,F}\sqsubset \underline{U}_i\]
such that
\[\underline{U}^{j}_{i,F}\cap \underline{U}_i^c = \underline{U}^{c,j}_{i,F},\qquad \underline{U}^{j}_{i,G}\cap\underline{U}_i^c = \underline{U}^{c,j}_{i,G}.\]
Such sets $\underline{U}^{j}_{i,F}, \underline{U}^{j}_{i,G}$ can be built by putting a $\Gamma$-invariant metric $d_i$ on $U_i$ and considering the normal bundles of $\underline{U}^{c,j}_{i,F}, \underline{U}^{c,j}_{i,G}$ in $U_i$. (Note that $d_i$ does not come from an admissible metric as in Section \ref{atlases} because $d_i$ and $d_j$ are not chosen compatibly.)

Then the sets $\underline{\psi}_i(\underline{s}_i^{-1}(0)\cap\underline{U}^{0}_{i,F})$ are precompact open subsets of $F_i$ that intersect $X_c$ at exactly $F_i^{c,0}$. Additionally, together they cover some $2\delta>0$ neighborhood of $X_c$.

Let $\{F_i\}$ be the open footprint cover of $X$ coming from $\mathcal{K}$. We now aim to use Lemma \ref{coverreductionlem} to find the desired cover reduction of $\{F_i\}$. Let $H_i^0\ldots H_i^k\sqsubset L_i^{k+1}\sqsubset\ldots F_i$ be a series of nested covers as in (\ref{711}). Let
\[0<\lambda_{2N}<\cdots<\lambda_{0}<2\delta\] Now form the sets
\[S_i^j := \big(H_i^j\setminus\overline{B_{\lambda_{2j}}(X_c)}\big)\cup\underline{\psi}_I\left(\underline{s}_I^{-1}(0)\cap \underline{U}_{i,F}^j\right)\]
\[R_i^j := \big(L_i^j\setminus\overline{B_{\lambda_{2j-1}}(X_c)}\big)\cup\underline{\psi}_I\left(\underline{s}_I^{-1}(0)\cap \underline{U}_{i,G}^j\right).\]
Here we use the convention that $\underline{U}^{j}_{i,F}=\underline{U}^{j}_{i,G}=\emptyset$ if $i\notin\{1,\ldots,N^c\}$. Then these form a nested reduction $S_i^0\ldots S_i^k\sqsubset R_i^{k+1}\sqsubset\ldots F_i$ as in (\ref{711}). Additionally we have the intersections
\[S_i^j\cap X_c = F_i^{c,j},\qquad R_i^j\cap X_c = G_i^{c,j}.\]
Therefore, we can form
\[Z_I := \left(\bigcap_{i\in I}R_i^{|I|}\right)\setminus\bigcup_{j\notin I}\overline{S_j^{|I|}}\]
which by Lemma \ref{coverreductionlem} is a cover reduction of $\{F_i\}$ and by construction has the property that $Z_I\cap X_c = Z_I^c$. Therefore, $\{Z_I\}$ is the desired cover reduction.
\end{proof}

Proposition \ref{reduction} says that any tame atlas has a reduction, but \cite[Proposition 5.3.5]{mwtop} proves the following stronger statement which we will use.

\begin{proposition}\cite[Proposition 5.3.5.]{mwtop}\label{reductionstrong}
Let $\mathcal{K}$ be a tame Kuranishi atlas with footprint cover $\{F_i\}$. Let $\{Z_I\}$ be a cover reduction. Then $\mathcal{K}$ has a reduction $\mathcal{V}$ such that the reduced footprint cover is $\{Z_I\}$.
\end{proposition}

We now use Lemma \ref{footprintreduction} and Proposition \ref{reductionstrong} to prove Lemma \ref{treduction}.

\begin{proof}[Proof of Lemma \ref{treduction}]
Let $\{Z_I^c\}, \{Z_I\}$ be cover reductions of $\{F_i^c\}, \{F_i\}$ as in Lemma \ref{footprintreduction}. Use Proposition \ref{reductionstrong} to construct a reduction $\mathcal{V}_c$ of $\mathcal{K}_c$ with reduced footprint cover $\{Z_I^c\}$. We now aim to construct the domains $V_I$ of the desired reduction $\mathcal{V}$ of $\mathcal{K}$.

Let $\underline{P}_I\sqsubset \underline{U}_I$ be an open, precompact subset such that
\[\underline{P}_I\cap \underline{U}_I^c = \underline{V}_I^c,\qquad \overline{\underline{P}_I}\cap \underline{U}_I^c = \overline{\underline{V}_I^c}.\]
Together the sets $\{\underline{\psi}_I(\underline{s}_I^{-1}(0)\cap \underline{P}_I)\}$ cover some $\delta>0$ neighborhood of $X_c$.

The sets $\left\{Z_I\cap\big(X\setminus B_{\delta}(X_c)\big)\right\}$ form a cover reduction of the cover $\left\{F_i\cap\big(X\setminus B_{\delta}(X_c)\big)\right\}$ of $X\setminus B_{\delta}(X_c)$. We then use Lemma \ref{514} to provide precompact open sets $\underline{Q}_I\sqsubset \underline{U}_I\setminus \underline{U}_I^c$ such that
\[\underline{Q}_I\cap \underline{s}_I^{-1}(0) = \underline{\psi}_I^{-1}\Big(Z_I\cap\big(X\setminus B_{\delta}(X_c)\big)\Big),\]
\[\overline{\underline{Q}_I}\cap \underline{s}_I^{-1}(0) = \underline{\psi}_I^{-1}\Big(\overline{Z_I}\cap\big(X\setminus B_{\delta}(X_c)\big)\Big).\]

Let $\underline{d}_I$ be an admissible metric on $\underline{U}_I$ (see Section \ref{atlases}) and $\varepsilon>0$ be small enough so that the $\varepsilon$-neighborhood $B_\varepsilon(\underline{V}_I^c)$ of $\underline{V}_I^c$ in $\underline{U}_I$ is precompact in $\underline{U}_I$. Define
\[\underline{W}_I:=\underline{Q}_I\cup \big(\underline{P}_I\cap B_\varepsilon(\underline{V}_I^c)\big).\]
It is easy to see that these satisfy the following:
\[\underline{W}_I\sqsubset \underline{U}_I,\qquad \underline{W}_I\cap \underline{s}_I^{-1}(0) = \underline{\psi}_I^{-1}(Z_I),\]
\[\overline{\underline{W}_I}\cap \underline{s}_I^{-1}(0) = \underline{\psi}_I^{-1}(\overline{Z_I}),\qquad \underline{W}_I\cap \underline{U}_I^c = \underline{V}_I^c.\]
Thus, the sets $\underline{\mathcal{W}} = \{\underline{W}_I\}$ satisfy conditions $(\ref{reduction1})$ and $(\ref{reduction3})$ of Definition \ref{reductiondef} and have the desired intersection property $\underline{W}_I\cap \underline{U}_I^c = \underline{V}_I^c$. The proof of Proposition \ref{reduction} in \cite{mwtop} contains a procedure to construct a reduction $\underline{\mathcal{V}} = \{\underline{V}_I\}$ of $\underline{\mathcal{K}}$ from a $\underline{\mathcal{W}}$ satisfying conditions $(\ref{reduction1})$ and $(\ref{reduction3})$. So to complete the proof of Lemma \ref{treduction} we just need to check that the resulting $\underline{V}_I$ satisfies $\underline{V}_I\cap \underline{U}_I^c = \underline{V}_I^c$. Then lifting the reduction $\underline{\mathcal{V}}$ of $\underline{\mathcal{K}}$ to a reduction $\mathcal{V}$ of $\mathcal{K}$ will complete the proof. We will now summarize the aforementioned procedure to construct $\underline{\mathcal{V}}$. Define
\[C(I) := \{J\in\mathcal{I}_{\mathcal{K}}~|~I\subset J\textnormal{ or } J\subset I\},\]
and for each $J\notin C(I)$ define
\[Y_{IJ} := \overline{\underline{W}_I}\cap\pi_{\underline{\mathcal{K}}}^{-1}\left(\pi_{\underline{\mathcal{K}}}(\overline{\underline{W}_J})\right)\]
and let $B(Y_{IJ})\subset \underline{U}_I$ be a closed neighborhood of $Y_{IJ}$ for each $J\notin C(I)$ such that
\[B(Y_{IJ})\cap \underline{\psi}_I^{-1}(\overline{Z_I})=\emptyset.\]
Then \cite{mwtop} proves that
\[\underline{V}_I:= \underline{W}_I\setminus\bigcup_{J\notin C(I)}B(Y_{IJ})\]
is a reduction. Thus, in order to complete the proof of Lemma \ref{treduction}, we just need to check that $B(Y_{IJ})\cap \underline{U}_I^c=\emptyset$. To see this, note that $\pi_{\underline{\mathcal{K}}}\big(\underline{P}_I\cap B_\varepsilon(\underline{V}_I^c)\big)$ lies in a $\varepsilon$-neighborhood of $\pi_{\underline{\mathcal{K}_c}}(\overline{\underline{V}_I^c})\subset |\mathcal{K}|$. So if $\varepsilon>0$ is chosen small enough
\[\pi_{\underline{\mathcal{K}}}\big(\underline{P}_I\cap B_\varepsilon(\underline{V}_I^c)\big)\cap\pi_{\underline{\mathcal{K}}}\big(\underline{P}_I\cap B_\varepsilon(\underline{V}_I^c)\big)\neq\emptyset\quad\Rightarrow\quad\pi_{\underline{\mathcal{K}_c}}(\overline{\underline{V}_I^c})\cap\pi_{\underline{\mathcal{K}_c}}(\overline{\underline{V}_J^c})\neq\emptyset\quad\Rightarrow\quad I\subset J\textnormal{ or } J\subset I.\]
Therefore, $Y_{IJ}\cap \underline{U}_I^c=\emptyset$ and hence we can choose $B(Y_{IJ})$ such that $B(Y_{IJ})\cap \underline{U}_I^c=\emptyset$.
\end{proof}

The next lemma states that given a compatible transverse reduction, one can construct a perturbation section compatible with the transverse subatlas. Its proof is contained in Section \ref{tsections}. The precise definition of an \textit{adapted perturbation} is given in Section \ref{originalsection}. In particular, an adapted perturbation is admissible, transverse, and precompact as in Definition \ref{perturbdef}.

\begin{lemma}\label{tsection}
Let $\mathcal{K}_c$ be a tame transverse subatlas of a tame atlas $\mathcal{K}$.  Let $\mathcal{C}\sqsubset\mathcal{V}$ be nested reductions for $\mathcal{K}$ and $\mathcal{C}_c\sqsubset\mathcal{V}_c$ nested reductions for $\mathcal{K}_c$ such that $\mathcal{C},\mathcal{C}_c$ and $\mathcal{V},\mathcal{V}_c$ are compatible transverse reductions. Let $0<\delta<\delta_{\mathcal{V}_c}$ and $0<\sigma<\sigma(\delta,\mathcal{V},\mathcal{C})$. Then there exists a $(\mathcal{V},\mathcal{C},\delta,\sigma)$-adapted perturbation $\nu$ of $s_{\mathcal{K}}|_{\mathcal{V}}$ such that $\nu_c := \nu|_{\mathcal{V}_c}$ is a $(\mathcal{V}_c,\mathcal{C}_c,\delta,\sigma)$-adapted perturbation of $s^c_{\mathcal{K}_c}|_{\mathcal{V}_c}$.
\end{lemma}

We can now combine Lemmas \ref{ttaming}, \ref{treduction}, and \ref{tsection} to prove Proposition \ref{tprop}.

\begin{proof}[Proof of Proposition \ref{tprop}]
By Lemma \ref{ttaming}, we may assume that $\mathcal{K}_c$ is a tame transverse subatlas of $\mathcal{K}$. We then apply Lemma \ref{treduction} to get a compatible transverse reduction $\mathcal{V}_c$ and $\mathcal{V}$. Finally, we apply Lemma \ref{tsection} to get compatible perturbation sections $\nu_c$ and $\nu$. This will produce corresponding perturbed zero sets $|\mathbf{Z}^{\nu_c}| = |(s_{\mathcal{K}_c}|_{\mathcal{V}_c}^{\setminus\Gamma}+\nu_c)^{-1}(0)|$ and $|\mathbf{Z}^{\nu}|=|(s_{\mathcal{K}}|_\mathcal{V}^{\setminus\Gamma}+\nu)^{-1}(0)|$. Moreover, each branch $Z_I=(s_I|_{V_I}+\nu_I)^{-1}(0)$ of $|\mathbf{Z}^{\nu}|$ comes equipped with an evaluation map $f$. It follows from Lemma \ref{treduction} that the domains of the reduction $\mathcal{V}_c$ are cut out transversally by the evaluation map from the domains of $\mathcal{V}$. Each branch $Z_I^c=(s_I^c|_{V_I^c}+\nu_I^c)^{-1}(0)$ of $|\mathbf{Z}^{\nu_c}|$ is cut out transversally by $s_I^c|_{V_I^c}+\nu_I^c$, each $Z_I$ is cut out transversally by $s_I|_{V_I}+\nu_I$, and $(s_I+\nu_I)|_{f^{-1}(C)} = s_I^c+\nu_I^c$. In particular, $\dim Z_I^c = \dim Z_I - $ codim $N$. Therefore, on each branch $f|_{Z_I}:Z_I\rightarrow N$ is transverse to $C\subset N$ and
\[f|_{Z_I}^{-1}(C) = Z_I^c.\]
Therefore, using Proposition \ref{wbm}, $|\mathbf{Z}^{\nu_c}|_{\mathcal{H}}$ is a compact weighted branched submanifold of $|\mathbf{Z}^{\nu}|_{\mathcal{H}}$ with each branch cut out transversally by the evaluation map.
\end{proof}

We now aim to apply Proposition \ref{tprop} to the atlas $\mathcal{K}_c$ on $X_c = \overline{\mathcal{M}}_{0,k}(A,J;c)$. However, $\mathcal{K}_c$ is a $C^1$ SS atlas, so we cannot directly apply Proposition \ref{tprop}. Nonetheless, Proposition \ref{tprop} continues to hold in the $C^1$ SS category with appropriate definitions.
\begin{definition}\label{c1tdef}
Let $X_c$ be a closed subset of a metric space $X$. Let $\mathcal{K}$ be a $C^1$ SS atlas on $X$. An atlas $\mathcal{K}_c$ on $X_c$ is called a \textbf{$\mathbf{C^1}$ SS transverse subatlas} if the conditions of Definition \ref{tdefinition} for being a transverse subatlas hold in the $C^1$ SS category. That is we require a $C^1$ SS evaluation map transverse (in the usual sense) to a submanifold $C$ and the domains $U_I^c$ of $\mathcal{K}_c$ are $C^1$ SS submanifolds of the domains of $\mathcal{K}$.
\end{definition}
\begin{proposition}\label{c1tprop}
Let $\mathcal{K}_c$ be a $C^1$ SS transverse subatlas of $\mathcal{K}$. Then the conclusion of Proposition \ref{tprop} holds.
\end{proposition}
\begin{proof}
The proof of Proposition \ref{tprop} relies on Lemmas \ref{ttaming}, \ref{treduction}, and \ref{tsection}. Lemmas \ref{ttaming} and \ref{treduction} hold in the $C^1$ SS category with no changes to their proofs. Lemma \ref{tsection} also holds in the $C^1$ SS category; for a discussion of this see Lemma \ref{c1transverse}. With these three lemmas, the proof of Proposition \ref{tprop} carries though in the $C^1$ SS category.
\end{proof}

We can apply Proposition \ref{c1tprop} to the atlas $\mathcal{K}_c$ on $X_c = \overline{\mathcal{M}}_{0,k}(A,J;c)$. To do this, we first show that $\mathcal{K}_c$ is a $C^1$ SS transverse subatlas of an atlas $\mathcal{K}$ on $X$.
\begin{lemma}\label{xcit}
Let $\mathcal{K}_c$ be a Kuranishi atlas for
\[X_c = \overline{\mathcal{M}}_{0,k}(A,J;c)\] as described in Section \ref{homconstruction}, so $\mathcal{K}_c$ is obtained by constraining the evaluation map $ev_k$ by a submanifold representative $Z_c$ of $c$ from (\ref{cdef}).  Then there is a $C^1$ SS Kuranishi atlas $\mathcal{K}$ for $X=\overline{\mathcal{M}}_{0,k}(A,J)$ such that $\mathcal{K}_c$ is a $C^1$ SS transverse subatlas of $\mathcal{K}$ with $ev_k$ as the evaluation map $f$ required in Definition \ref{tdefinition}.
\end{lemma}
\begin{proof}
The charts for $\mathcal{K}_c$ are constructed by choosing center points $\{[\Sigma_i,\mathbf{z}_i,f_i]\in X_c\}$ and building charts whose footprint contains this center point. The charts for $\mathcal{K}$ are also constructed in the same manner; by construction the chart for $\mathcal{K}_c$ centered at $[\Sigma_i,\mathbf{z}_i,f_i]\in X_c$ is a subset of the chart for $\mathcal{K}$ centered at $\{[\Sigma_i,\mathbf{z}_i,f_i]\in X_c\}$ cut out transversally by the evaluation map $ev_k$. Furthermore, by making the obstruction bundle sufficiently large, we can assume $ev_k$ is SS transverse to the submanifold representative $Z_c\subset M^k$ of $c\in H_{*}(M^k)$ from (\ref{cdef}). So forming the atlas $\mathcal{K}$ by first choosing center points $\{[\Sigma_i,\mathbf{z}_i,f_i]\in X_c\}$ to cover a neighborhood of $X_c$ and then choosing points outside a neighborhood of $X_c$, one can choose the domains of $\mathcal{K}$ to satisfy Definition \ref{tdefinition}. The fact that $ev_k$ is SS transverse to $Z_c$ implies that the domains of $\mathcal{K}_c$ are $C^1$ SS submanifolds and hence satisfy Definition \ref{c1tdef}.
\end{proof}

\begin{proof}[Proof of Theorem \ref{homologyagree}]
Let $\mathcal{K}_c$ be a Kuranishi atlas for $X_c$. By Lemma \ref{xcit}, there is a $C^1$ SS Kuranishi atlas $\mathcal{K}$ for $X$ such that $\mathcal{K}_c$ is a $C^1$ SS transverse subatlas of $\mathcal{K}$. It suffices to choose these particular atlases because atlases constructed in this manner are cobordant and by \cite{mwtop, mwfund, mwiso} the resulting invariants are well defined. We can then apply Proposition \ref{c1tprop} to construct $|\mathbf{Z}^{\nu_c}|_{\mathcal{H}}$ as a compact weighted branched submanifold of $|\mathbf{Z}^{\nu}|_{\mathcal{H}}$ with each branch cut out transversally by the evaluation map $ev_k$. The virtual fundamental class $[X_c]_{\mathcal{K}_c}^{vir}$ is constructed from $|\mathbf{Z}^{\nu_c}|_{\mathcal{H}}$. On the other hand $[X]_{\mathcal{K}}^{vir}$ is constructed from $|\mathbf{Z}^{\nu}|_{\mathcal{H}}$ and $|\mathbf{Z}_{\nu_c}^c|_{\mathcal{H}}$ is cut out transversally by $ev_k:|\mathbf{Z}^{\nu}|_{\mathcal{H}}\rightarrow M^k$. Therefore, $(ev_k)_{*}([X]_{\mathcal{K}}^{vir})\cdot (c_1\times\cdots\times c_k)$ agrees with $[X_c]_{\mathcal{K}_c}^{vir}$. This proves Theorem \ref{homologyagree}.
\end{proof}

\subsection{Domain constraints}\label{domainconstruction}
One can also consider Gromov-Witten invariants where the homology type of the (stabilized) domain is constrained. Formally, this is the same problem as for constraints $c\in H_{*}(M)$ because, as will be discussed, the invariants are calculated by a homological intersection with $\beta\in H_{*}(\overline{\mathcal{M}}_{0,k})$. However, the stratification of $\overline{\mathcal{M}}_{0,k}$ must be taken into account due to the fact that $\overline{\mathcal{M}}_{0,k}$ is considered as a $C^1$ SS (not smooth) manifold $\overline{\mathcal{M}}_{0,k}^{new}$ as in Proposition \ref{dm1}. Therefore, we begin this section by first discussing the homology of Deligne-Mumford space, before later considering Gromov-Witten invariants with constrained domains.

First, we will describe a convenient generating set for $H_{*}(\overline{\mathcal{M}}_{0,k})$ and two choices of cycles, $Y_\beta $ and $Z_\beta$, representing these generators that will be useful later in this section and in Section \ref{gwaxioms}. For a set $I\subset\{1,\ldots,k\}$, $|I|\geq 3$, define the homology class $\beta_{k,I}\in H_{*}(\overline{\mathcal{M}}_{0,k})$ by
\[\beta_{k,I} = PD(\pi_{k,I}^{*}PD([pt]))\in H_{2k-2|I|}(\overline{\mathcal{M}}_{0,k}),\]
where $\pi_{k,I}:\overline{\mathcal{M}}_{0,k}\rightarrow\overline{\mathcal{M}}_{0,|I|}$ is the map which forgets the marked points not belonging to the set $I$. Next we describe a submanifold $Y_{k,I}$ that represents the class $\beta_{k,I}$. Write $I=:\{i_0,\ldots,i_\ell\}$ with $i_0<\cdots<i_\ell$. Let $T$ be a tree with one node and define $\mathbf{y}_I\in\overline{\mathcal{M}}_{0,|I|}$ to be the unique element corresponding to the $I$-labelled tree modelled on $T$ with marked points $i_j=j\in S^2 = \mathbb{C}\cup\{\infty\}$. So $\mathbf{y}_I$ is a distinguished element in the open top stratum $\mathcal{M}_{0,k}\subset\overline{\mathcal{M}}_{0,k}$.
\begin{lemma}\cite[Lemma 7.5.5\footnote{This lemma is stated in \cite{JHOL} but its proof contains an error. They use a different cycle representative $Z_{k,I} = \pi_{k,I}^{-1}(\mathbf{z}_I)$ of $\beta_{k,I}$ that is described later in Lemma \ref{zki}. Although the cycle they describe does represent the class $\beta_{k,I}$, they incorrectly state that $\mathbf{z}_{I}$ is a regular value of $\pi_{k,I}$ and that $\pi_{k,I}^{-1}(\mathbf{z}_{I})$ is a smooth submanifold. This in turn makes their proof of $(ii)$ incorrect. However, the statements are correct and the proof below is taken virtually verbatim with only a slight change to correct this error.}]{JHOL}\label{755}
Let $I\subset\{1,\ldots,k\}$ be such that $|I|\geq 3$.
\begin{enumerate}[(i)]
\item\label{7551} The class $\beta_{k,I}$ may be represented by the $C^1$ SS submanifold $Y_{k,I}:=\pi_{k,I}^{-1}(\mathbf{y}_I)\subset\overline{\mathcal{M}}_{0,k}$.
\item\label{7552} Denote by $\pi_{0,k}:\overline{\mathcal{M}}_{0,k}\rightarrow\overline{\mathcal{M}}_{0,k-1}$ the map which forgets the last marked point. Then $\beta_{k,I} = PD(\pi_{0,k}^{*}PD(\beta_{k-1,I}))$ when $k\notin I$. Moreover,
    \[(\pi_{0,k})_{*}\beta_{k,I} = \begin{cases}
    \beta_{k-1,I\setminus\{k\}},& k\in I,|I|\geq4\\
    0,& k\notin I.
    \end{cases}\]
\end{enumerate}
\end{lemma}
\begin{proof}
To prove $(\ref{7551})$, we note that $\pi_{k,I}$ is a $C^1$ SS map and hence by Lemma \ref{sstransverselemma} it suffices to check that $\mathbf{y}_I$ is a regular value of $\pi_{k,I}$. To do this it suffices to show that for each $\mathbf{z}\in Y_{k,I}$ there are local coordinates near $\mathbf{y}_I$ in $\overline{\mathcal{M}}_{0,|I|}$ that extend to a system of local coordinates near $\mathbf{z}$. This is clearly true by the description of coordinate charts for $\overline{\mathcal{M}}_{0,k}$ and $\overline{\mathcal{M}}_{0,|I|}$ via cross ratios as described in \cite[Appendix D]{JHOL}.

To prove the first assertion of $(\ref{7552})$, assume first that $k\notin I$. Then $\pi_{k,I}=\pi_{k-1,I}\circ\pi_{0,k}$ and the assertion that $\beta_{k,I} = PD(\pi_{0,k}^{*}PD(\beta_{k-1,I}))$ follows from functoriality.

Assume next that $k\in I$ and $|I|\geq 4$. Clearly $\pi_{0,k}(\mathbf{y}_I) = \mathbf{y}_{I\setminus\{k\}}$ and hence $\pi_{0,k}(Y_{k,I}) \subset Y_{k-1,I\setminus\{k\}}$. It is also easily seen that the map $\pi_{0,k}:Y_{k,I}\rightarrow Y_{k-1,I\setminus\{k\}}$ is onto. The submanifolds $Y_{k,I}$ and $Y_{k-1,I\setminus\{k\}}$ carry a stratification induced from $\overline{\mathcal{M}}_{0,|I|}$ and $\mathcal{\mathcal{M}}_{0,|I|-1}$. Let $Y_{k,I}^\circ,Y_{k-1,I\setminus\{k\}}^\circ$ denotes the open top stratum of $Y_{k,I},Y_{k-1,I\setminus\{k\}}$ respectively. The top stratum consists of elements modelled on the tree with one node. All other strata have codimension at least two. It is easily seen that $\pi_{0,k}$ maps $Y_{k,I}^\circ$ to $Y_{k-1,I\setminus\{k\}}^\circ$. Next take $[\mathbf{z}]\in Y_{k-1,I\setminus\{k\}}^\circ$. The element $[\mathbf{z}]$ has a unique preimage $\pi_{0,k}^{-1}([\mathbf{z}])\in Y_{k,I}$. To see this, there is a unique representative $\mathbf{z}$ of $[\mathbf{z}]$ with the marked points $z_{i_0},\ldots,z_{i_{\ell-1}}$ at $0,\ldots,\ell-1\in S^2$. If there is no marked point at $\ell$, then the unique lift $\pi_{0,k}^{-1}([\mathbf{z}])\in Y_{k,I}$ is obtained from $\mathbf{z}$ by adding a $k$th marked point at $\ell$. If there is a marked point $z_i$ with $i\notin I$ at $\ell$, then $\pi_{0,k}^{-1}([\mathbf{z}])\in Y_{k,I}$ is obtained from $\mathbf{z}$ by adding a bubble at $\ell$ with marked points $z_i$ and $z_k$. Therefore, $\pi_{0,k}:Y_{k,I}\rightarrow Y_{k-1,I\setminus\{k\}}$ is injective on $\pi_{0,k}^{-1}(Y_{k-1,I\setminus\{k\}}^\circ)$. In summary, $\pi_{0,k}$ is onto and is a holomorphic diffeomorphism on a subset of $Y_{k,I}$ with complement of codimension at least two. Therefore $(\pi_{0,k})_{*}([Y_{k,I}]) = [Y_{k-1,I\setminus\{k\}}]$ which is the second assertion of $(\ref{7552})$.

Assume finally that $k\notin I$ or $|I|\leq 3$. Then $\pi_{0,k}$ maps $Y_{k,I}$ to a manifold of strictly lower dimension. This proves the final assertion of $(\ref{7552})$.
\end{proof}

The representatives $Y_{k,I}$ described in Lemma \ref{755} prove useful in many circumstances and are the easiest to work with because they $C^1$ SS submanifolds formed by preimages of points in the top stratum of $\overline{\mathcal{M}}_{0,k}$. However, in the context of Gromov-Witten theory one is naturally lead to consider nodal representatives as well. So whereas $Y_{k,I}$ is formed as the preimage of the smoothest points, we describe another cycle representative $Z_{k,I}$ of $\beta_{k,I}$ that is formed as the preimage of the most nodal points.\footnote{These are the representatives considered in \cite{JHOL}.} Write $I=\{i_0,\ldots,i_\ell\}$. Let $T$ be the tree with $\ell-1$ vertices $\alpha_1,\ldots,\alpha_{\ell-1}$ such that there is an edge from $\alpha_i$ to $\alpha_{i+1}$. Let $\mathbf{z}_I\in\overline{\mathcal{M}}_{0,|I|}$ be the unique element with the first two marked points on $\alpha_1$, the last two marked points on $\alpha_{\ell-1}$ and one marked point on every other vertex in increasing order. Thus, $\mathbf{z}_{I}$ is the unique element with cross ratios $w_{i_{j_0}i_{j_1}i_{j_2}i_{j_3}}=\infty$ for $0\leq j_0<j_1<j_2\leq\ell$.
\begin{lemma}\label{zki}
The class $\beta_{k,I}$ can be represented by the cycle $Z_{k,I}:=\pi_{k,I}^{-1}(\mathbf{z}_I)$.
\end{lemma}
\begin{proof}
The point $\mathbf{z}_I$ is no longer a regular value of $\pi_{k,I}$. However, it the limit point of sequence $\{\mathbf{z}_m\}_{m\in\mathbb{N}}$ of points where $\mathbf{z}_m$ is in the top stratum of $\overline{\mathcal{M}}_{0,|I|}$. Each $\mathbf{z}_m$ is a regular value of $\pi_{k,I}$ by the reasoning in the proof of Lemma \ref{755}. Hence $\pi_{k,I}^{-1}(\mathbf{z}_I)$ represents the same homology class as $\pi_{k,I}^{-1}(\mathbf{z}_m)$, which is $\beta_{k,I}$.
\end{proof}

The classes $\beta_{k,I}$ are of interest to us due to the following lemma.

\begin{lemma}\label{generate}
The classes $\beta_{k,I}$ generate $H_{*}(\overline{\mathcal{M}_{0,k}})$.
\end{lemma}
\begin{proof}
For a stable labelled tree $T$, let $\overline{\mathcal{M}}_{0,T}\subset \overline{\mathcal{M}}_{0,k}$ denote the closure of the set of stable curves modelled on $T$. The set $\overline{\mathcal{M}}_{0,T}$ is a smooth submanifold and in \cite{keel} Keel proves that all such submanifolds generate the homology of $\overline{\mathcal{M}}_{0,k}$. Therefore, it suffices to prove that for every $T$, $\overline{\mathcal{M}}_{0,T}=Z_{k,I}$ for some $k,I$ where $Z_{k,I}$ is a cycle representing the class $\beta_{k,I}$. To see this, let $T'$ be the stable labelled tree obtained from $T$ by removing the maximum number of marked points such that $T'$ is still stable. Let $I:= \{i_0,\ldots,i_\ell\}\subset\{1,\ldots,k\}$ be the indexing set of marked points remaining in $T'$. Define $\mathbf{w}_I\in\overline{\mathcal{M}}_{0,|I|}$ to be the unique element that corresponds to the stable labelled tree $T'$. Let $W_{k,I} = \pi_{k,I}^{-1}(\mathbf{w}_I)$. It is clear that $\overline{\mathcal{M}}_{0,T} = W_{k,I}$. The fact that $W_{k,I}$ represents the class $\beta_{k,I}$ follows from the same reasoning as the proof of Lemma \ref{zki}.
\end{proof}

We will now use the description of $H_{*}(\overline{\mathcal{M}}_{0,k})$ above to construct Gromov-Witten invariants with homological constraints on the domain. As with constraints on the image as considered in Section \ref{homconstruction}, there are two different ways one can build invariants with domain constraints. Let $\beta\in H_{*}(\overline{\mathcal{M}}_{0,k})$. The space $X=\overline{\mathcal{M}}_{0,k}(A,J)$ carries a forgetful map $\pi_0:\overline{\mathcal{M}}_{0,k}(A,J)\rightarrow\overline{\mathcal{M}}_{0,k}$ which takes a stable map to its underlying stabilized domain. We wish to consider elements of $X$ whose (stabilized) domain is constrained. Define the Gromov-Witten invariant
\begin{equation}
GW_{A,k}^M(\beta) = (\pi_0)_{*}([X]_{\mathcal{K}}^{vir})\cdot\beta.
\end{equation}
The atlas $\mathcal{K}$ has dimension $d:=2n+2c_1(A)+2k-6$, so the invariant is $0$ unless $2k-6-\deg\beta=d$.

Alternatively, as with homological constraints on image, we could directly constrain the space $X$. Let $Y_{\beta_{k,I}}\subset \overline{\mathcal{M}}_{0,k}$ be the $C^1$ SS submanifold in Lemma \ref{755} representing the homology class $\beta_{k,I}$. Consider the subset of $X$
\begin{equation}\label{xbeta}
X_{\beta_{k,I}} = \{[\Sigma,\mathbf{z},f]\in\overline{\mathcal{M}}_{0,k}(A,J)~|~\pi_0([\Sigma,\mathbf{z},f])\in Y_{\beta_{k,I}}\}.
\end{equation}
If $d$ is the formal dimension of $X$, then $X_{\beta_{k,I}}$ has formal dimension $d-(2k-6-\deg\beta_{k,I})=0$. One can then construct a weak atlas $\mathcal{K}_{\beta_{k,I}}$ for $X_{\beta_{k,I}}$. We do this using Gromov Witten charts as described in Section \ref{gwcharts}. The domain $\widehat{U}$ carries a $C^1$ SS forgetful map $\pi_0$ to $\overline{\mathcal{M}}_{0,k}^{new}$ that takes an element to its underlying stabilized domain (see Proposition \ref{dm1}). We can define $\widehat{U}_{\beta_{k,I}}$ to be an open neighborhood of $(\vec{0},\mathbf{w}_0,\mathbf{z}_0,f_0)$ in
\begin{equation}\label{ubeta}
\widehat{U}_{\beta_{k,I}} := \left\{(\vec{e},[\Sigma,\mathbf{w},\mathbf{z},f])\in\widehat{U}~|~\pi_0(f)\in Y_{\beta_{k,I}}\right\}.
\end{equation}
To ensure that this is manifold of dimension $\dim (\widehat{U}) - (2k-6-\deg\beta_{k,I})$ we need that the forgetful map $\pi_0:\widehat{U}\rightarrow\overline{\mathcal{M}}_{0,k}^{new}$ is SS transverse to $Y_{\beta_{k,I}}$. The obstruction space $E$ is chosen large enough so that this is true for sufficiently small $\widehat{U}$. We then obtain a weak Kuranishi atlas, which has a well defined class $[X_{\beta_{k,I}}]_{\mathcal{K}}^{vir}\in \check{H}_{0}(X_{\beta_{k,I}};\mathbb{Q})$ which is a rational number. We expect that this is the same as the Gromov-Witten invariant $GW_{A,k}^M(\beta_{k,I})$. The next theorem says that this is true.
\begin{theorem}\label{homologyagree2}
The two approaches to defining Gromov-Witten invariants with domain constraints as described above agree. That is for all $\beta_{k,I}$
\[GW_{A,k}^M(\beta_{k,I}) := (\pi_0)_{*}([X]_{\mathcal{K}}^{vir})\cdot \beta_{k,I} = [X_{\beta_{k,I}}]_{\mathcal{K}_{\beta_{k,I}}}^{vir}\in\mathbb{Q}.\]
\end{theorem}
It suffices to consider the classes $\beta_{k,I}$ due to Lemma \ref{generate}. We prove Theorem \ref{homologyagree2} exactly as we proved Theorem \ref{homologyagree}. We first have the following lemma.
\begin{lemma}\label{xbit}
Let $\mathcal{K}_{\beta_{k,I}}$ be a Kuranishi atlas for
\[X_{\beta_{k,I}} :=\{[\Sigma,\mathbf{z},f]\in\overline{\mathcal{M}}_{0,k}(A,J)~|~ \pi_0([\Sigma,\mathbf{z},f]) \in Y_{\beta_{k,I}}\}.\] Then there is a $C^1$ SS Kuranishi atlas $\mathcal{K}$ for $X=\overline{\mathcal{M}}_{0,k}(A,J)$ such that $\mathcal{K}_{\beta_{k,I}}$ is a $C^1$ SS transverse subatlas of $\mathcal{K}$ with $\pi_0$ as the evaluation map $f$ required in Definition \ref{tdefinition}.
\end{lemma}
\begin{proof}
The proof of this lemma is easy and follows from the same reasoning as the proof of Lemma \ref{xcit}. The only difference is that in Lemma \ref{xcit} the domains were cut out by a SS transverse evaluation map while in the present situation, the domains are cut out transversally and are locally modelled by $\mathbb{R}^k\times\mathbb{C}^{\underline{\ell}}\times\{0\}$. Thus, they are also $C^1$ SS submanifolds (see Remark \ref{sssubrem}) which is all that is required to apply Proposition \ref{c1tprop}.
\end{proof}
\begin{proof}[Proof of Theorem \ref{homologyagree2}]
Theorem \ref{homologyagree2} is proven exactly as Theorem \ref{homologyagree}, with Lemma \ref{xbit} being used in place of Lemma \ref{xcit}.
\end{proof}

Combining the definitions of this section and Section \ref{homconstruction} allows us to define the Gromov-Witten invariant
\begin{equation}\label{defofgw}
GW^M_{A,k}(c_1,\ldots,c_k;\beta):= (ev_k\times\pi_0)_{*}([X]^{vir})\cdot (c_1\times\cdots\times c_k\times\beta)
\end{equation}
where $c_i\in H_{*}(M), \beta\in H_{*}(\overline{\mathcal{M}}_{0,k}),$ and $\cdot$ denotes the intersection product in $M^k\times\overline{\mathcal{M}}_{0,k}$. In order to establish the Gromov-Witten axioms in Section \ref{gwaxioms} we will prove Theorem \ref{inttheorem} below, a statement that combines Theorems \ref{homologyagree} and \ref{homologyagree2}. It roughly says that all ways of cutting down the moduli space by some subset of the homological constraints and then pushing forward and intersecting will result in the same invariant. To be more precise, let $0\leq\ell\leq k$ and $Z_{c_1\times\cdots\times c_\ell}$ be a submanifold of $M^\ell$ representing the class $c_1\times\cdots\times c_\ell$. Let $Y_{\beta_{k,I}}$ be the submanifold of $\overline{\mathcal{M}}_{0,k}$ representing the class $\beta_{k,I}$ from Lemma \ref{755}. Define
\begin{equation}\label{M1}
\overline{\mathcal{M}}_{0,k}(A,J;c_1,\ldots,c_\ell) = \{[\Sigma,\mathbf{z},f]\in\overline{\mathcal{M}}_{0,k}(A,J)~|~ev_\ell(f)\in Z_{c_1\times\cdots\times c_\ell}\}
\end{equation}
\begin{equation}\label{M2}
\overline{\mathcal{M}}_{0,k}(A,J;c_1,\ldots,c_\ell;\beta_{k,I}) = \{[\Sigma,\mathbf{z},f]\in\overline{\mathcal{M}}_{0,k}(A,J)~|~ev_\ell(f)\in Z_{c_1\times\cdots\times c_\ell}, \pi_0(f)\in Y_{\beta_{k,I}}\}
\end{equation}
Here $ev_\ell$ denotes the evaluation map for the first $\ell$ marked points. Kuranishi atlases for these spaces can be built using the techniques described in this section and Section \ref{homconstruction}. Let $ev_{k-\ell}$ denote the evaluation map for the last $k-\ell$ marked points.
\begin{theorem}\label{inttheorem}
For any $0\leq\ell\leq k$,
\[GW^M_{A,k}(c_1,\ldots,c_k;\beta) = (ev_{k-\ell}\times\pi_0)_{*}\big(\big[\overline{\mathcal{M}}_{0,k}(A,J;c_1,\ldots,c_\ell)\big]^{vir}\big)\cdot(c_{\ell+1}\times\cdots\times c_k\times \beta)\]
\[GW^M_{A,k}(c_1,\ldots,c_k;\beta_{k,I}) = (ev_{k-\ell})_{*}\big(\big[\overline{\mathcal{M}}_{0,k}(A,J;c_1,\ldots,c_\ell;\beta_{k,I})\big]^{vir}\big)\cdot(c_{\ell+1}\times\cdots\times c_k).\]
In particular, in the case $\ell=k$,
\[GW^M_{A,k}(c_1,\ldots,c_k;\beta_{k,I}) = [\overline{\mathcal{M}}_{0,k}(A,J;c_1,\ldots,c_k;\beta_{k,I})]^{vir} \in \mathbb{Q}.\]
\end{theorem}
\begin{proof}
This theorem is proved in the same way as Theorems \ref{homologyagree} and \ref{homologyagree2}.
\end{proof}

\begin{remark}[Non-SS domains]\label{nonsmooth}
It will be useful in Section \ref{gwaxioms} to use the cycles $Z_{k,I}$ defined in Lemma \ref{zki} to define the space $X_{\beta_{k,I}}$ defined by (\ref{xbeta}). When building an atlas $\mathcal{K}_{\beta_{k,I}}$ for $X_{\beta_{k,I}}$, the spaces $\widehat{U}_{\beta_{k,I}}$ as defined by (\ref{ubeta}) (from which the domains of $\mathcal{K}_{\beta_{k,I}}$ are formed) are no longer smooth manifolds, which we required in Definition \ref{chart}. However we would still like to use $Z_{k,I}$ for geometric considerations.

The domains $\widehat{U}_{\beta_{k,I}}$ are not smooth, but have a simple local description since they have local charts of the form
\[\left\{(x,z_1,\ldots,z_n)\in\mathbb{R}^k\times\mathbb{C}^n~|~z_j=0, j\in J\subset\{1,\ldots,n\}\right\}.\]
We call spaces with such local models \textbf{stratified pseudomanifolds}. Stratified pseudomanifolds have a natural stratification induced from the stratification of $\mathbb{R}^k\times\mathbb{C}^{\underline{n}}$ described in Example \ref{stratex}. However, they are obviously not manifolds but rather locally a finite union of manifolds.

The taming and reduction procedures carry through for stratified pseudomanifold domains because they hold in the much more general setting of weak filtered topological atlases. The construction of a perturbation section can also be done for stratified pseudomanifold domains by constructing the section inductively on strata. This is very similar to the construction of the perturbation section in Proposition \ref{c1section}, see Remark \ref{pseudoremark}.
\end{remark}

\section{Proof of Gromov-Witten axioms}\label{gwaxioms}
The section is dedicated to the proof of Theorem \ref{axiomthm}.

\subsection{Proof of Effective, Symmetry, Grading, Homology, and Zero axioms}
We first prove the \textit{(Effective), (Symmetry), (Grading), (Homology),} and \textit{(Zero)} axioms, which follow easily from definitions.

The \textit{(Effective)} axiom follows from the fact that when $\omega(A)<0$, the moduli space $X=\overline{\mathcal{M}}_{0,k}(A,J)$ is empty and hence $[X]_{\mathcal{K}}^{vir}=0$. \hfill{$\Diamond$}\\

The \textit{(Symmetry)} axioms follows from the fact that the symmetric group $S_k$ acts on the moduli space $X$ by permuting the marked points and the evaluation map $ev_k:\overline{\mathcal{M}}_{0,k}(A,J)\rightarrow M^k$ and the forgetful map $\pi_0:\overline{\mathcal{M}}_{0,k}(A,J)\rightarrow\overline{\mathcal{M}}_{0,k}$ are equivariant under this action. The same is true for a weak Kuranishi atlas on $X$. \hfill{$\Diamond$}\\

The \textit{(Grading)} axiom follows from the dimension formula for $X$. \hfill{$\Diamond$}\\

The \textit{(Homology)} axiom follows by defining
\[\sigma_{A,k} = (ev_k\times\pi_0)_{*}([X]_{\mathcal{K}}^{vir}).\]
As defined in (\ref{defofgw}),
\[GW^M_{A,k}(c_1,\ldots,c_k;\beta) = (c_1\times\cdots\times c_k\times\beta)\cdot\sigma_{A,k}.\]\hfill{$\Diamond$}

For the \textit{(Zero)} axiom, we first note that the moduli space $X=\overline{\mathcal{M}}_{0,k}(0,J)$ is naturally diffeomorphic to the manifold $M\times\overline{\mathcal{M}}_{0,k}$ because every $J$-holomorphic sphere in homology class $0$ is constant. Note that this is a manifold and is cut out transversally because every constant $J$-holomorphic sphere is regular by \cite[Lemma 6.7.6]{JHOL}. Therefore, we can take a Kuranishi atlas $\mathcal{K}$ for $X$ with one chart $\mathbf{K}$ with domain $M\times\overline{\mathcal{M}}_{0,k}$, trivial obstruction bundle, no isotopy group, and identity as the footprint map. Then $[X]_\mathcal{K}^{vir}$ is the usual fundamental class of a manifold. The evaluation map
\[ev_k\times\pi_0:\overline{\mathcal{M}}_{0,k}(0,J)=M^k\times\overline{\mathcal{M}}_{0,k}\rightarrow M^k\times\overline{\mathcal{M}}_{0,k}\]
is given by $(ev_k\times\pi_0)(x,C) = (x,\ldots,x,C)$ where $x\in M,C\in\overline{\mathcal{M}}_{0,k}$. Therefore,
\begin{align*}
    GW^M_{0,k}(c_1,\ldots,c_k;\beta) &= \big((ev_k)_*([X]_{\mathcal{K}}^{vir})\times(\pi_0)_*([X]_{\mathcal{K}}^{vir})\big)\cdot\big((c_1\times\cdots\times c_k)\times\beta\big)\\
    &=\big(\Delta\cdot(c_1\times\cdots\times c_k)\big)\times([\overline{\mathcal{M}}_{0,k}]\cdot\beta)
\end{align*}
where $\Delta$ denotes the diagonal in $M^k$. The number $[\overline{\mathcal{M}}_{0,k}]\cdot\beta$ is nonzero only if $\deg(\beta)=0$. In the case $\beta=$[pt],
\begin{align*}
    GW^M_{0,k}(c_1,\ldots,c_k;[pt]) &=\big(\Delta\cdot(c_1\times\cdots\times c_k)\big)\\
    &=c_1\cap\cdots\cap c_k.
\end{align*}
This proves the \textit{(Zero)} axiom.\hfill{$\Diamond$}

\subsection{Proof of Fundamental class, Divisor, and Splitting axioms}
We now prove the \textit{(Fundamental class), (Divisor),} and \textit{(Splitting)} which have more geometric content and require greater care than the previously considered axioms.

For the \textit{(Fundamental class)} axiom, by Lemma \ref{generate}, it suffices to prove the axiom for $\beta=\beta_{k,I}$. Therefore, we must prove that
\begin{equation}\label{fundcase0}
GW^M_{A,k}(c_1,\ldots,c_{k-1},[M];\beta_{k,I}) = GW^M_{A,k-1}(c_1,\ldots,c_{k-1};(\pi_{0,k})_*\beta_{k,I})
\end{equation}
where
\[\pi_{0,k}:\overline{\mathcal{M}}_{0,k}\rightarrow\overline{\mathcal{M}}_{0,k-1}\]
is the map that forgets the last marked point. We then break the proof into three cases.

\noindent\underline{Case 1:} $k\in I$ and $|I|\geq 4$.

In this case, by Lemma \ref{755}, (\ref{fundcase0}) becomes
\begin{equation}\label{fundcase1}
GW^M_{A,k}(c_1,\ldots,c_{k-1},[M];\beta_{k,I}) = GW^M_{A,k-1}(c_1,\ldots,c_{k-1};\beta_{k-1,I\setminus\{k\}}).
\end{equation}

We first consider the case when $k \in I$ and $|I|\geq 4$. Write $I=\{i_0,\ldots,i_{\ell},k\}\subset\{1,\ldots,k\}$. By Theorem \ref{inttheorem} in the case $\ell=k$, we can construct both invariants in (\ref{fundcase1}) by constraining the moduli space using all homological constraints so that the relevant atlases have dimension $0$. That is, using the notation of (\ref{M2})
\[GW^M_{A,k}(c_1,\ldots,c_{k-1},[M];\beta_{k,I}) = \big[\overline{\mathcal{M}}_{0,k}(A,J;c_1,\ldots,c_{k-1},[M];\beta_{k,I})\big]^{vir}_{\mathcal{K}_k}\]
\[GW^M_{A,k-1}(c_1,\ldots,c_{k-1};\beta_{k-1,I\setminus\{k\}}) = \big[\overline{\mathcal{M}}_{0,k-1}(A,J;c_1,\ldots,c_{k-1};\beta_{k-1,I\setminus\{k\}})\big]^{vir}_{\mathcal{K}_{k-1}},\]
where we define these moduli spaces by picking a submanifold $Z_{c_1\times\cdots\times c_{k-1}}$ to represent the class $c_1\times\cdots\times c_{k-1}$ and choose $Z_{c_1\times\cdots\times c_{k-1}}\times M$ to represent $c_1\times\cdots\times c_{k-1}\times [M]$. We can also use the $C^1$ SS submanifolds $Y_{k,I},Y_{k-1,I\setminus\{k\}}$ from Lemma \ref{755} to represent $\beta_{k,I},\beta_{k-1,I\setminus\{k\}}$. The $0$-dimensional atlases $\mathcal{K}_k$ and $\mathcal{K}_{k-1}$ can be built as described in Sections \ref{homconstruction} and \ref{domainconstruction}.

Consider the following subset $V_{k-1}$ of the domain $U_{k-1}$ of a chart in $\mathcal{K}_{k-1}$. Elements of $V_{k-1}$ are those satisfying the following conditions:
\begin{enumerate}[(1)]
\item The domains are smooth.
\item Fix a parametrization such that the marked points $z_{i_0},\ldots,z_{i_{k-1}}$ are at $0,\ldots,\ell\in S^2$ (this is possible by the definition of $\beta_{k-1,I\setminus\{k\}}$). No other marked point in $\mathbf{z}$ or $\mathbf{w}$ is at $\ell+1\in S^2$.
\end{enumerate}
Both of these conditions are codimension two phenomena in $U_{k-1}$, so by Proposition \ref{codim2}, we can build the perturbation functor $\nu_{k-1}$ for $\mathcal{K}_{k-1}$ such that the perturbed zero set lies in the domains $V_{k-1}$.

We can build the atlases $\mathcal{K}_{k-1}$ and $\mathcal{K}_k$ by using the same slicing manifold. Then the domains of $\mathcal{K}_k$ carry a forgetful map $\pi_{0,k}$ to the domains of $\mathcal{K}_{k-1}$ which forgets the $k$th marked point. This map is onto and injective over the preimage $\pi_{0,k}^{-1}(V_{k-1})$. The proofs of these facts follows by the same reasoning as the proof of Lemma \ref{755}. We can tame $\mathcal{K}_{k-1}$ and use preimages under $\pi_{0,k}$ to tame $\mathcal{K}_k$. We can also define the perturbation functor $\nu_k$ for $\mathcal{K}_k$ by pulling pack $\nu_{k-1}$ via $\pi_{0,k}$. Transversality of $\nu_{k-1}$ implies transversality for $\nu_k$  and by construction, the forgetful map induces a homeomorphism from the perturbed zero set in $\mathcal{K}_k$ to the perturbed zero set in $\mathcal{K}_{k-1}$ by $\pi_{0,k}$. Therefore \[\big[\overline{\mathcal{M}}_{0,k}(A,J;c_1,\ldots,c_{k-1},[M];\beta_{k,I})\big]^{vir}_{\mathcal{K}_k} = \big[\overline{\mathcal{M}}_{0,k-1}(A,J;c_1,\ldots,c_{k-1};\beta_{k-1,I\setminus\{k\}})\big]^{vir}_{\mathcal{K}_{k-1}}\]
which proves (\ref{fundcase1}).\\

\noindent\underline{Case 2:} $k\notin I$.

In this case, by Lemma \ref{755}, (\ref{fundcase0}) becomes
\begin{equation}\label{fundcase2}
GW^M_{A,k}(c_1,\ldots,c_{k-1},[M];\beta_{k,I}) = 0.
\end{equation}

We will use Theorem \ref{inttheorem} to build the invariant $GW^M_{A,k}(c_1,\ldots,c_{k-1},[M];\beta_{k,I})$ using a $0$-dimensional atlas $\mathcal{K}_k$ on the space $\overline{\mathcal{M}}_{0,k}(c_1,\ldots,c_{k-1},[M];\beta_{k,I})$. Now consider the space $\displaystyle \overline{\mathcal{M}}_{0,k-1}(c_1,\ldots,c_{k-1},\beta_{k-1,I})$. This space carries an atlas $\mathcal{K}_{k-1}$ that can be made by using the same choice of slicing manifold as $\mathcal{K}_k$. The map $\pi_{0,k}$ that forgets the last marked point induces a functor $\pi_{0,k}:\mathcal{K}_k\rightarrow\mathcal{K}_{k-1}$. Over the top stratum of $\mathcal{K}_{k-1}$ this map has smooth fibres diffeomorphic to $S^2$ because the last marked point is unconstrained. Hence me may tame $\mathcal{K}_{k-1}$ and use the inverse images under $\pi_{0,k}$ to tame $\mathcal{K}_k$. Note that $\mathcal{K}_{k-1}$ is of negative dimension (its domains have dimension two less than those for $\mathcal{K}_{k}$ and they have the same obstruction bundle) so we can construct the perturbation section $\nu_{k-1}$ for $\mathcal{K}_{k-1}$ so that $s_{\mathcal{K}_{k-1}}+\nu_{k-1}$ has no zeros. We can then define the perturbation section $\nu_k$ by pulling back $\nu_{k-1}$ via $\pi_{0,k}$. Hence this perturbed zero set is also empty and $[\overline{\mathcal{M}}_{0,k}(c_1,\ldots,c_{k-1},[M];\beta_{k,I})]^{vir}_{\mathcal{K}_k}=0$. This proves $(\ref{fundcase2})$ when $k\notin I$.\\

\noindent\underline{Case 3:} $|I|=3$.
In this case, $\beta_{k,I}=[\overline{\mathcal{M}}_{0,3}]=pt$ and hence (\ref{fundcase0}) becomes
\begin{equation}\label{fundcase3}
GW^M_{A,k}(c_1,\ldots,c_{k-1},[M];\beta_{k,I}) =  0,
\end{equation}
just as in Case 2. Note that we can still define $\overline{\mathcal{M}}_{0,2}(A,J)$ even though $\overline{\mathcal{M}}_{0,2}$ is not defined. The fact that $(\ref{fundcase3})$ holds follows by the same reasoning as Case 2.$\hfill{\Diamond}$\\

For the \textit{(Divisor)} axiom, by Lemmas \ref{755} and \ref{generate} it suffices to prove
\begin{equation}\label{divisorcase}
GW^M_{A,k}(c_1,\ldots,c_k;\beta_{k,I}) = (c_k\cdot A)~GW^M_{A,k-1}(c_1,\ldots,c_{k-1};\beta_{k-1,I})
\end{equation}
where $k\notin I$. By Theorem \ref{inttheorem} in the case $\ell=k-1$, we can construct the invariants in (\ref{divisorcase}) in the following manner (using the notation of (\ref{M2}))
\[GW^M_{A,k-1}(c_1,\ldots,c_{k-1};\beta_{k-1,I}) = [\overline{\mathcal{M}}_{0,k-1}(A,J;c_1,\ldots,c_{k-1};\beta_{k-1,I})]^{vir}_{\mathcal{K}_{k-1}}\]
\[GW^M_{A,k}(c_1,\ldots,c_k;\beta_{k,I}) = ev_{*}\big([\overline{\mathcal{M}}_{0,k}(A,J;c_1,\ldots,c_{k-1};\beta_{k,I})]^{vir}_{\mathcal{K}_k}\big)\cdot c_k.\]
Here $ev$ denotes the evaluation map for the $k$th marked point. We can use the same submanifold representative of the class $c_1\times\cdots\times c_{k-1}$ for both moduli spaces and the submanifolds $Y_{k,I},Y_{k-1,I}$ from Lemma \ref{755} to represent the classes $\beta_{k,I},\beta_{k-1,I}$. The atlases $\mathcal{K}_k$ and $\mathcal{K}_{k-1}$ are $2$ and $0$-dimensional respectively. Therefore, to prove the \textit{(Divisor)} axiom it suffices to show
\begin{align}
\nonumber &ev_{*}\big([\overline{\mathcal{M}}_{0,k}(A,J;c_1,\ldots,c_{k-1};\beta_{k,I})]^{vir}_{\mathcal{K}_k}\big)\\ \label{eachA}=&\big([\overline{\mathcal{M}}_{0,k-1}(A,J;c_1,\ldots,c_{k-1};\beta_{k-1,I})]^{vir}_{\mathcal{K}_{k-1}}\big)A \in H_2(M)
\end{align}
where here we are regarding $[\overline{\mathcal{M}}_{0,k-1}(A,J;c_1,\ldots,c_{k-1};\beta_{k-1,I})]^{vir}_{\mathcal{K}_{k-1}}\in\mathbb{Q}$. Additionally, by Proposition \ref{codim2} we can build the virtual class $[\overline{\mathcal{M}}_{0,k-1}(A,J;c_1,\ldots,c_{k-1};\beta_{k-1,I})]^{vir}_{\mathcal{K}_{k-1}}$ from a perturbed zero set, all of whose elements have smooth domain.

Consider the forgetful map $\pi_{0,k}:\overline{\mathcal{M}}_{0,k}\rightarrow \overline{\mathcal{M}}_{0,k-1}$ that forgets the $k$th marked point. This map restricts to a map $\pi_{0,k}:Y_{k,I}\rightarrow Y_{k-1,I}$ that behaves like the universal curve. That is for $[C]\in Y_{k-1,I}$, $\pi_{0,k}^{-1}([C])$ can be identified with the curve $[C]$; the last marked point in $Y_{k,I}$ is unconstrained. In the same way, we have a forgetful functor on Kuranishi atlases $\pi_{0,k}:|\mathcal{K}_k|\rightarrow|\mathcal{K}_{k-1}|$. Once again, the last marked point is unconstrained, so the fibre over any element of $|\mathcal{K}_{k-1}|$ is exactly the domain of that element. In other words $\mathcal{K}_k$ is constructed in the same way as $\mathcal{K}_{k-1}$ but with an additional unconstrained marked point $z_k$. Therefore, the perturbed zero set used to construct $[\overline{\mathcal{M}}_{0,k-1}(A,J;c_1,\ldots,c_{k-1};\beta_{k-1,I})]^{vir}_{\mathcal{K}_{k-1}}$ is a finite number of weighted points and for each point $[C]$ in the perturbed zero set we get a two dimensional zero set in $\mathcal{K}_{k}$ that is the domain of $[C]$ and has the same weight as $[C]$. By construction $[C]$ is smooth and hence the two dimensional zero set is cut out transversally. Thus, the pushforward of this zero set by the evaluation map $ev$ is the image of the element $[C]$ and so is in the class $A$. This proves (\ref{eachA}) and hence the \textit{(Divisor)} axiom. $\hfill\Diamond$\\

For the \textit{(Splitting)} axiom let us examine the product map
\[\phi_S:\overline{\mathcal{M}}_{0,k_0+1}\times\overline{\mathcal{M}}_{0,k_1+1}\rightarrow\overline{\mathcal{M}}_{0,k}\]
that appears in the \textit{(Splitting)} axiom. Let $I_0,I_1\subset\{1,\ldots,k\}$ be disjoint subset such that $|I_0|\geq 2$ and $|I_1|\geq 2$. Let $\mathcal{S}(I_0,I_1)$ denote the set of all splittings $\mathcal{S}=(\mathcal{S}_0,\mathcal{S}_1)$ of the index set $\{1,\ldots,k\}$ such that $I_0\subset\mathcal{S}_0$ and $I_1\subset \mathcal{S}_1$. Given a splitting
\[\mathcal{S}=(\mathcal{S}_0,\mathcal{S}_1)\in\mathcal{S}(I_0,I_1),\]
let $k_i:=|\mathcal{S}_i|$. Let
\[\sigma_0:\{1,\ldots,k_0\}\rightarrow\{1,\ldots,k\},\qquad \sigma_1:\{2,\ldots,k_1+1\}\rightarrow\{1,\ldots,k\}\]
be the unique order preserving injections such that im $\sigma_i=\mathcal{S}_i$. Define subsets $I_{\mathcal{S}_i}\subset\{1,\ldots,k_i\}$
\[I_{\mathcal{S}_0}:=\sigma_0^{-1}(I)\cup\{k_0+1\},\qquad I_{\mathcal{S}_1}:=\{1\}\cup\sigma_1^{-1}(I).\]
The following lemma is adapted from \cite[Lemma 7.5.9]{JHOL}. Its proof is exactly the same as the proof in \cite{JHOL} and is an easy consequence of the definition of $\phi_S$.
\begin{lemma}\cite[Lemma 7.5.9]{JHOL}\label{759}
Let $Z_{k,I}$ be the cycle representing the homology class $\beta_{k,I}$ as described in Lemma \ref{zki}. Then
\[Z_{k,I} = \bigcup_{\mathcal{S}\in\mathcal{S}(I_0,I_1)}\phi_S(Z_{k_0+1,I_{\mathcal{S}_0}}\times Z_{k_1+1,I_{\mathcal{S}_1}})\]
and hence
\[\beta_{k,I} = \sum_{\mathcal{S}\in\mathcal{S}(I_0,I_1)}{\phi_S}_{*}(\beta_{k_0+1,I_{\mathcal{S}_0}}\otimes\beta_{k_0+1,I_{\mathcal{S}_0}}).\]
\end{lemma}
Using Lemma \ref{generate} we can rephrase the \textit{(Splitting)} axiom as the following.
\begin{align}\label{splitting}
&GW^M_{A,k}(c_1,\ldots,c_k;\beta_{k,I})\\
\nonumber=&\varepsilon(S,c)\sum_{\substack{A_0+A_1\\=A}}\sum_{\nu,\mu}GW^M_{A_0,k_0+1}(\{c_i\}_{i\in S_0},e_\nu;\beta_{k_0+1,I_{\mathcal{S}_0}})g^{\nu\mu}GW^M_{A_1,k_1+1}(e_\mu,\{c_j\}_{j\in S_1};\beta_{k_1+1,I_{\mathcal{S}_1}}).
\end{align}
We can use Theorem \ref{inttheorem} to build the invariants of $(\ref{splitting})$ by constraining the moduli spaces using all homological constraints so that the relevant moduli spaces have dimension $0$ and further we can use the submanifolds $Z_{k,I}$ from Lemma \ref{759}. That is, more explicitly
\begin{align*}
GW^M_{A,k}(c_1,\ldots,c_k;\beta_{k,I})&=\big[\overline{\mathcal{M}}_{0,k}(A,J;c_1,\ldots,c_k;Z_{k,I})\big]^{vir}_{\mathcal{K}}\\
GW^M_{A_0,k_0+1}(\{c_i\}_{i\in S_0},e_\nu;\beta_{k_0+1,I_{\mathcal{S}_0}})&=\big[\overline{\mathcal{M}}_{0,k}(A_0,J;\{c_i\}_{i\in S_0},e_\nu;Z_{k_0+1,I_{\mathcal{S}_0}})\big]^{vir}_{\mathcal{K}_0}\\
GW^M_{A_1,k_1+1}(e_\mu,\{c_j\}_{j\in S_1};\beta_{k_1+1,I_{\mathcal{S}_1}})&=\big[\overline{\mathcal{M}}_{0,k}(A_1,J;e_\mu,\{c_i\}_{i\in S_1};Z_{k_1+1,I_{\mathcal{S}_1}})\big]^{vir}_{\mathcal{K}_1}.
\end{align*}

By the discussion of Remark \ref{nonsmooth}, although the use of the $Z_{k,I}$ create domains that are stratified pseudomanifolds, all necessary results still hold. The proof of the \textit{(Splitting)} axiom is then just as the proof of \textit{(Splitting)} axiom in \cite{JHOL}.\hfill$\Diamond$

\section{Constructing Perturbation Sections}\label{sections}
In the construction of the virtual fundamental class associated to a weak Kuranishi atlas, the most intricate step and the one that relies most on the structure of the atlas is the construction of the perturbation section. As such, the key step in many of the results of previous sections was the construction of a perturbation compatible with the variation under consideration. In this section we outline the construction of the perturbation section in its simplest form, smooth and in the absence of isotropy, following \cite{mwfund}. We then describe the necessary variations needed for the applications of this paper.

\subsection{Perturbation sections}\label{originalsection}
This section contains the basic construction on which subsequent ones will be based. We begin with some background and then make precise statements about the existence of a perturbation section. Finally, we outline its construction. There are only two places in the construction where smoothness is used and they are marked by ($\ast$) and ($\ast\ast$). These will be the places in which modifications for other constructions are required. It is also important to understand how the transversality of the perturbation section is established so that we can adapt to the relative case (Section \ref{tsections}).

Given a reduction $\mathcal{V}$ of a tame Kuranishi atlas $\mathcal{K}$ with admissible metric $d$, define $\delta_\mathcal{V}>0$ to be the maximal constant such that any $\delta<\delta_\mathcal{V}$ satisfies
\begin{equation}
B^I_{2\delta}(V_I) \sqsubset U_I\qquad \forall I\in\mathcal{I}_{\mathcal{K}}
\end{equation}
\begin{equation}\label{reddisjoint}
B_{2\delta}^I(\pi_{\mathcal{K}}(\overline{V_I}))\cap B^I_{2\delta}(\pi_{\mathcal{K}}(\overline{V_J}))\neq\emptyset\quad\Rightarrow\quad I\subset J\textnormal{ or }J\subset I.
\end{equation}
Here $B_\delta^I$ denotes a $\delta$-metric neighborhood in $U_I$. Then for $\delta<\delta_{\mathcal{V}}$, we can define a decreasing sequence
\[V_I^k := B^I_{2^{-k}\delta}(V_I)\sqsubset U_I,\qquad \textnormal{for }k\geq0.\]
These define nested reductions $\mathcal{V}^k:=\{V_I^k\}\sqsupset \mathcal{V}^{k+1}$. Keeping with the notation of \cite{mwfund}, in the absence of isotropy the projection $\rho_{IJ}:\widetilde{U}_{IJ}\rightarrow U_{IJ}$ is invertible and we define
\[\phi_{IJ}=\rho^{-1}_{IJ}.\]
The sets
\[N_{JI}^k:= V_J^k\cap\pi_{\mathcal{K}}^{-1}(\pi_{\mathcal{K}}(V_I^k)) = V_J^k\cap\phi_{IJ}(V^k_I\cap U_{IJ})\]
will play an important role in the construction of the perturbation. We abbreviate
\[N_J^k := \bigcup_{I\subsetneq J}N_{JI}^k\subset V_J^k\]
and call $N_J^{|J|}$ the \textbf{core} of $V_J^{|J|}$, since it is the part of the set on which there is some compatibility required between $\nu_J$ and $\nu_I$ for $I\subsetneq J$. When $|J|=k+1$ we will also need to consider \textbf{enlargements of the core} of the form $N_J^{k+\lambda}\supset N_J^{|J|}$, where $0<\lambda<1$.

Given nested reductions $\mathcal{C}\sqsubset \mathcal{V}$ of $(\mathcal{K},d)$ and $0<\delta<\delta_\mathcal{V}$ define
\begin{equation}\label{eta}
\eta_0:=(1-2^{-1/4})\delta,\qquad \eta_k := 2^{-k}\eta_0
\end{equation}
and
\[\sigma(\delta,\mathcal{V},\mathcal{C}):=\min_{J\in\mathcal{I}_\mathcal{K}}\inf\left\{\|s_J(x)\|~|~x\in\overline{V_J^{|J|}}\setminus\left(\widetilde{C}_J\cup\bigcup_{I\subsetneq J}B^J_{\eta_{|J|-\frac{1}{2}}}\left(N_{JI}^{|J|-\frac{1}{4}}\right)\right)\right\}\]
where
\[\widetilde{C}_J:=\bigcup_{J\subset K}\phi_{JK}^{-1}(C_K)\subset U_J.\]
It is proved in \cite{mwfund} that $\sigma(\delta,\mathcal{V},\mathcal{C})>0$. This quantity is a lower bound for $s_J$ away from the zero set and provides an upper bound for the permitted size of the perturbation.

The conditions that we need the perturbation section to satisfy are included in the following definition.
\begin{definition}\label{adaptedpert}
Given nested reductions $\mathcal{C}\sqsubset\mathcal{V}$ of a Kuranishi atlas $(\mathcal{K},d)$ and constants $0<\delta<\delta_\mathcal{V}$ and $0<\sigma\leq\sigma(\delta,\mathcal{V},\mathcal{C})$ we say that a perturbation $\nu:\mathbf{B}_{\mathcal{K}}|_{\mathcal{V}}\rightarrow\mathbf{E}_\mathcal{K}|_{\mathcal{V}}$ of $s_\mathcal{K}|_{\mathcal{V}}$ is $(\mathcal{V},\mathcal{C},\delta,\sigma)$-\textbf{adapted} if the following conditions hold for every $k=1,\ldots,\max_{I\in\mathcal{I}_\mathcal{K}}|I|$, where $\eta_k$ is as in (\ref{eta}).
\begin{enumerate}[(a)]
\item \label{conda} The perturbations are compatible on $\bigcup_{|I|\leq k}V_I^k$, that is
\[\nu_J|_{\widetilde{U}_{IJ}\cap V_J^k\cap\rho_{IJ}^{-1}(V_I^k)} = \widehat{\phi}_{IJ}\circ\nu_I\circ\rho_{IJ}|_{\widetilde{U}_{IJ}\cap V_J^k\cap \rho_{IJ}^{-1}(V_I^k)}\quad\forall I\subset J.\]
\item\label{condb} The perturbed sections are \textbf{transverse} in the sense that $(s_I|_{V_I^k}+\nu_I)\pitchfork 0$ for each $|I|\leq k$.
\item\label{condc} The perturbations are \textbf{strongly admissible} with radius $\eta_k$, that is for all $H\subsetneq I$ and such that $|I|\leq k$ we have
\[\nu_I\big(B_{\eta_k}^I(N_{IH}^k)\big)\subset\widehat{\phi}_{HI}(E_H).\]
In particular, this implies that there exists a neighborhood of the core such that any zero in this neighborhood is actually in the core itself.
\item\label{condd} The perturbed zero sets are controlled by $\mathcal{C}$ in the sense that $\pi_{\mathcal{K}}\big((s_I|_{V_I^k}+\nu_I)^{-1}(0)\big)\subset\pi_{\mathcal{K}}(\mathcal{C})$ for $|I|\leq k$. This is needed for compactness.
\item\label{conde} The perturbations are small enough to control the zero set in the sense that 
    \[\sup_{x\in V_I^k}\|\nu_I(x)\|<\sigma\]
    for $|I|\leq k$.
\end{enumerate}
\end{definition}

We first note that adapted perturbation are automatically admissible, precompact, and transverse in the sense of Definition \ref{perturbdef} (see \cite[Remark 3.3.2]{mwiso}). The zero set is cut out transversally by $(\ref{condb})$ and compact by $(\ref{condd})$ and $(\ref{conde})$. The main result, Proposition \ref{section} below, is that such perturbations always exist.
\begin{proposition}\cite[Proposition 7.3.7]{mwfund},\cite[Proposition 3.3.3]{mwiso}\label{section}
Let $\mathcal{K}$ be a Kuranishi atlas with nested reductions $\mathcal{C}\sqsubset\mathcal{V}$. For all $0<\delta<\delta_{\mathcal{V}}$ and $0<\sigma<\sigma(\delta,\mathcal{V},\mathcal{C})$ there exists a $(\mathcal{V},\mathcal{C},\delta,\sigma)$-\textit{adapted perturbation} $\nu$ of $s_\mathcal{K}|_\mathcal{V}$.
\end{proposition}

\textbf{Overview of the proof of Proposition \ref{section}}
The intricate aspects of this construction are carried out in \cite{mwfund} in the case of no isotropy. Then \cite{mwiso} shows how to adapt this construction in the presence of isotropy; this is discussed in Section \ref{atlases}. We will focus here on the case of no isotropy to better illuminate the construction. We will use the definitions and notation from Section \ref{atlases} as well as earlier in this section. One can also refer to \cite{notes} for a more detailed discussion than this paper, but less involved than \cite{mwfund, mwiso}.

The construction of the section is by an inductive process that constructs the required perturbations $\nu_I$. We will induct on $k=|I|$ and at the $k$th step construct functions $\nu_I:V_I^{|I|}\rightarrow E_I$ for all $|I|=k$ that, together with the $\nu|_{V_I^k}$ for $|I|<k$ obtained by restriction from earlier steps, satisfy conditions $(a)$-$(e)$ of Definition \ref{adaptedpert}. The key point is that because the different sets $V_I^{|I|}, |I|=k$ are disjoint by (\ref{reddisjoint}), at the $k$th inductive step, the functions $\nu_I$ can be constructed independently.

The inductive hypothesis is that suitable $\nu_I:V^{|I|}\rightarrow E_I$ have been found for $|I|\leq k$. We then need to construct $\nu_J$ for some $|J|=k+1$. We construct $\nu_J$ as a sum $\nu_J =\widetilde{\nu}_J+\nu_\pitchfork$ where
\begin{itemize}
\item $\widetilde{\nu}_J|_{N_J^{k+1}} = \mu_J|_{N_J^{k+1}}$ where $\mu_J:N_J^k\rightarrow E_J$ is defined on the enlarged core by using the compatibility conditions.
\item$\nu_\pitchfork$ is a final small perturbation chosen so as to achieve transversality.
\end{itemize}
Constructing an extension $\widetilde{\mu}_J$ of $\mu_J$ is the intricate portion of the procedure and $\widetilde{\nu}_J$ is obtained from $\widetilde{\mu}_J$ using a cutoff function. To construct $\widetilde{\mu}_J$, each component $\widetilde{\mu}_J^j:N_J^k\rightarrow E_j$, for $j\in J$, is constructed separately. To do this, it suffices to construct $\widetilde{\mu}_J^j$ on a neighborhood
\[W_J:=B_{\eta_k+\frac{1}{2}}^J(N_J^{k+\frac{1}{2}}).\]
This construction is done by another inductive process. These $\widetilde{\mu}_J^j$ are required to satisfy some vanishing and size conditions so that $(a)$-$(e)$ are satisfied.

At this point, the problem is further reduced to constructing extensions $\widetilde{\mu}_z$ near each point $z$ in a set $B'_L$ defined in \cite[Equation (7.3.27)]{mwfund}. In most cases this choice of extension is trivial, that is it can be chosen to be zero or is already given by the compatibility conditions.
\begin{itemize}[($\ast$)]
\item The only nontrivial extension that is required is when $z\in\overline{N_{JL}^{k+\frac{1}{2}}}$. Recall that $\overline{N_{JL}^{k+\frac{1}{2}}}$ is a compact subset of the smooth manifold $N_{JL}^k$. Thus, we can choose a small $r_z>0$ such that $B_{r_z}^J(z)$ lies in a submanifold chart for $N_{JL}^k$. Then $\widetilde{\mu}_z$ is chosen to extend $\mu_J^j|_{B_{r_z}^J(z)\cap N_{JL}^k}$ and be constant in the normal direction. This is one of the places where the smooth structure of $\mathcal{K}$ is used.
\end{itemize}
Here are some additional important remarks on the construction of $\mu_J$.
\begin{enumerate}[(A)]
\item The constants $\eta_k$ are chosen so that $W_J\cap N_J^k\sqsubset N_J^{k+\frac{1}{4}}$. Further, $B_{\eta_k+\frac{1}{2}}^J(N_{JI}^{k+\frac{3}{4}})\cap N_{JI}^k=N_{JI}^{k+\frac{1}{2}}$. Together with tameness, this gives
    \[s_J^{-1}(E_I)\cap B_{\eta_k+\frac{1}{2}}^J(N_{JI}^{k+\frac{3}{4}})\subset N_{JI}^{k+\frac{1}{2}}\]
    which means that at each point in $\overline{B_{\eta_k+\frac{1}{2}}^J(N_{JI}^{k+\frac{1}{2}})}\setminus N_{JI}^k$ at least one component of $(s_J^j)_{j\in J\setminus I}$ is nonzero. This gives control over the zero sets as in the remarks below. It ensure that zeros near the core are actually in the core itself.
\item The following strong admissibility condition holds: If $I\subsetneq J$ and $j\in J\setminus I$, then $\widetilde{\mu}_J^j=0$ on $B_{\eta_k+\frac{1}{2}}^J(N_{JI}^{k+\frac{1}{2}})\subset W_J$ and on $N_{JI}^k$.
\item The section $s_J+\widetilde{\nu}_J$ is transverse to $0$ on $B:=B_{\eta_k+\frac{1}{2}}^J(N_{JI}^{k+\frac{3}{4}})\subset W_J$. This holds because by (A) and (B), all zeros in $B$ must lie in the subset $N_J^{k+\frac{1}{2}}$; but here the perturbation is pulled back from $\nu_I, I\subsetneq J$, so the zeros are transverse by the inductive hypothesis. The set $B$ compactly contains the neighborhood $B':=B_{\eta_k+1}^J(N_J^{k+1})$ of the core $N_J^{|J|}$ on which compatibility requires $\widetilde{\mu}_J=\mu_J=\widetilde{\nu}_J$.
\item For any perturbation $\widetilde{\nu}_J$ with support in $W_J$, $\|\widetilde{\nu}_J\|<\sigma$, and satisfying admissibility condition (C), we have
\[V_J^{k+1}\cap(s_J+\widetilde{\nu}_J)^{-1}(0)\subset N_J^{k+\frac{1}{4}}\cup(\widetilde{C}_J\setminus\overline{W}_J).\]
\end{enumerate}
At this stage of the construction $(a), (c), (d),$ and $(e)$ hold, so we only need to choose $\nu_\pitchfork$ to achieve transversality condition $(b)$ while keeping $\nu_\pitchfork$ small so that $(d)$ is still satisfied. To do this we first choose a relatively open neighborhood $W$ of $\overline{V_J^{k+1}}\setminus B$ in $\overline{V_J^{k+1}}$ so that
\[(s_J+\widetilde{\nu}_J)^{-1}(0)\cap W\subset \overline{V_J^{k+1}}\cap \widetilde{C}_J,\quad \overline{B'}\cap W=\emptyset.\]
This is possible by condition (D) and the fact $B'\sqsubset B\subset W_J$. By (C), transversality holds in $B$ and hence outside the compact subset $\overline{V_J^{k+1}}\setminus B\subset W$. Hence there is an open precompact subset $P\sqsubset W$ so that transversality holds on $W\setminus P$.
\begin{itemize}[($\ast\ast$)]
\item We then appeal to the Transversality Extension Theorem of \cite{gp} to choose $\nu_\pitchfork$ to be a very small smooth function with support in $P$ and values in $E_J$ so that $s_J+\widetilde{\nu}_J+\nu_\pitchfork$ is transverse on all of $W$ and hence $V_J^{k+1}$. This is the other place where the smooth structure of $\mathcal{K}$ is required.
\end{itemize}
This completes the inductive step.

\subsection{$C^1$ SS sections}\label{sectionsc1}
In this section, we adapt the construction of Section \ref{originalsection} to $C^1$ SS Kuranishi atlases. (See Section \ref{c1ss} for more about $C^1$ SS atlases.)
\begin{proposition}\label{c1section}
Let $\mathcal{K}$ be a $C^1$ SS Kuranishi atlas with nested reductions $\mathcal{C}\sqsubset\mathcal{V}$. For all $0<\delta<\delta_{\mathcal{V}}$ and $0<\sigma<\sigma(\delta,\mathcal{V},\mathcal{C})$ there exists a $(\mathcal{V},\mathcal{C},\delta,\sigma)$-adapted $C^1$ SS perturbation $\nu$ of $s_\mathcal{K}|_\mathcal{V}$ such that
\begin{itemize}
\item $s_I|_{V_I}+\nu_I:V_I\rightarrow E_I$ is SS transverse to $0$, i.e. its restriction to each stratum is transverse to $0$.
\end{itemize}
\end{proposition}

\begin{proof}[Proof of Proposition \ref{c1section}]
We use the same inductive process of the proof of Proposition \ref{section} and show that it can be done in the $C^1$ SS category with the additional property that it is SS transverse (see Definition \ref{sstransverse}). Note that $\widetilde{U}_{IJ} \subset s_J^{-1}(E_I)$ is a $C^1$ SS submanifold of $U_J$ by the discussion following Remark \ref{sssubrem}. As before, the section is constructed as a sum $\nu=\widetilde{\nu}+\nu_\pitchfork$. We first deal with the construction of $\widetilde{\nu}$. As described in the proof of Proposition \ref{section}, there are two analytic steps in the construction of $\widetilde{\nu}$, denoted by $(\ast)$ and $(\ast\ast)$. The first, $(\ast)$, comes in constructing extensions $\widetilde{\mu}_z$ when $z\in\overline{N_{JL}^{k+\frac{1}{2}}}$. In the proof of Proposition \ref{section}, we pick a small ball $B_{r_z}^J(z)$ such that $z$ is in a submanifold chart for $N_{JL}^k$, and then extend $\mu_J^j|_{B_{r_z}^J(z)\cap N_{JL}^k}$ to be constant in the normal directions. We will do the same thing in the $C^1$ SS category. Choose a projection $p:B_{r_z}^J(z)\rightarrow N_{JL}^K$ that is $C^1$, preserves strata, and is smooth when restricted to each stratum. As this is a local choice, it is clear that such a projection $p$ exists. Then extend $\mu_J^j|_{B_{r_z}^J(z)\cap N_{JL}^k}$ to be constant on the fibres of $p$. This defines a $C^1$ SS function which has the desired properties by the same reasoning as before.

The other step, $(\ast\ast)$, in the construction of $\nu$ is the final small smooth perturbation $\nu_\pitchfork$. To build $\nu_\pitchfork$ we appealed to the Transversality Extension Theorem of \cite{gp}. In the $C^1$ SS setting we can no longer directly apply this theorem because it relies on Sard's lemma, which when applied to a map $f:\mathbb{R}^n\rightarrow\mathbb{R}^m$ requires $C^k$ differentiability with $k\geq\max\{1,n-m+1\}$. Therefore, we need an appropriate analog in the $C^1$ SS category.

\begin{lemma}[$C^1$ SS Transversality Theorem]\label{c1tran}
Let $X$ be a $C^1$ SS manifold. Suppose $F:X\times S\rightarrow Y$ is a $C^1$ SS map, where $S,Y$ are smooth manifolds given the trivial stratification. Let $Z$ be a submanifold of $Y$. If $F$ is SS transverse to $Z$, then for almost every $s\in S$, $F(s,\cdot):X\rightarrow Y$ is SS transverse to $Z$.
\end{lemma}
\begin{proof}
The preimage $F^{-1}(Z)$ is a submanifold of $X\times S$. Moreover, since $F$ is SS transverse to $Z$, $F^{-1}(Z)$ carries a $C^1$ SS structure induced from that on $X\times S$. Let $\pi:X\times S\rightarrow S$ be the natural projection map. It is shown in \cite{gp} that if $s\in S$ is a regular value of the restricted projection $\pi:F^{-1}(Z)\rightarrow S$, then $F(s,\cdot)$ is transverse to $Z$. We can restrict the map $\pi$ to each stratum $T$ of $X$ to obtain smooth maps $\pi_T:F^{-1}(Z)_T\rightarrow S$. By Sard's theorem, almost every $s\in S$ is a regular value for $\pi_T$. Therefore, by taking a finite intersection, almost every $s\in S$ is a regular value of $\pi_T:F^{-1}(Z)_T\rightarrow S$ for all stratum $T$. Therefore, for almost every $s\in S$, $F(s,\cdot):X\rightarrow Y$ is SS transverse to $Z$.
\end{proof}

We can then use the $C^1$ SS Transversality Theorem to prove the following $C^1$ SS Extension Theorem, which is what is required for the construction of $\nu_\pitchfork$. Its proof, using Lemma \ref{c1tran}, is sketched below and taken virtually unchanged from \cite{gp}.

\begin{lemma}[$C^1$ SS Extension Theorem]
Let $X$ be a $C^1$ SS manifold. Let $C$ be a closed subset of $X$. Suppose $f:X\rightarrow Y$ is a $C^1$ SS map with $f$ SS transverse to $Z$ on $C$. Then there exists a $C^1$ SS map $g:X\rightarrow Y$ homotopic to $f$ such that $g$ is SS transverse to $Z$ and on a neighborhood of $C$ we have $f=g$.
\end{lemma}
\begin{proof}[Sketch of proof]
Let $F:X\times S\rightarrow Y$ parameterize small perturbations of the map $f$ where $S$ is taken to be a Euclidean ball in a large enough Euclidean space so that $F$ is SS transverse to $Z$. The $C^1$ SS Transversality Theorem implies that for almost every $s\in S$, the map $f_s(x):=F(x,s)$ is SS transverse to $Z$. Therefore, $f$ can always be homotoped to an SS transverse map. By using a cutoff function that is $0$ on a neighborhood of $C$ and $1$ on a larger neighborhood, we can ensure that this homotopy does not change $f$ along $C$.
\end{proof}

We can then apply the $C^1$ SS Extension Theorem to construct an arbitrarily small function $\nu_\pitchfork$ with the desired properties. This completes the construction of $\nu$ in the $C^1$ SS category.
\end{proof}

\begin{remark}\label{pseudoremark}
The key component of constructing a perturbation section for a $C^1$ SS atlas, as in the proof of Proposition \ref{c1section}, is to complete the inductive procedure of Section \ref{originalsection} compatible with the stratification. The final step is to achieve transversality with a small perturbation that is built on each strata. We note that both of these procedures can be done on \textit{stratified pesudomanifolds}, which were introduced in Remark \ref{nonsmooth}.
\end{remark}

\subsection{Transverse sections}\label{tsections}
\begin{proof}[Proof of Lemma \ref{tsection}]
We will use the notation in the proof of Proposition \ref{section} in Section \ref{originalsection}. Additionally, we will denote all of the subsets of $U_J^c$ involved in the construction of the perturbation section with a superscript ``$c$." For example,
\[N_{JI}^{c,k} := V_J^{c,k}\cap\pi_{\mathcal{K}_c}^{-1}(\pi_{\mathcal{K}_c}(V_I^{c,k})).\]
By the definition of a compatible transverse reduction, the core behaves well with respect to the transverse structure, that is
\[N_{JI}^k\cap U_J^c = N_{JI}^{c,k}.\]
We will use the same notation for metric balls. That is, for a subset $P\subset U_I^c$, we will denote $B^{c,I}_\delta(P)$ to denote the $\delta$-metric ball of $P$ in $U_I^c$. We can also view $P\subset U_I$ and we have the compatibility
\[B_\delta^I(P)\cap U_I^c = B_\delta^{c,I}(P)\subset U_I^c.\]
Using these facts as well as similar simple observations, we see that $\delta_{\mathcal{V}_c}\geq\delta_{\mathcal{V}}$ and $\sigma(\delta,\mathcal{V},\mathcal{C}) \leq \sigma(\delta,\mathcal{V}_c,\mathcal{C}_c)$. Using this, it is clear that the restriction of a $(\mathcal{V},\mathcal{C},\delta,\sigma)$-adapted perturbation $\nu$ to $\mathcal{V}_c$ satisfies conditions $(a), (c), (d)$ and $(e)$. Thus, we need to show that $\nu$ can be chosen so that its restriction to $\mathcal{V}_c$ satisfies condition $(b)$, that is we require $(s^c_I|_{V_I^{c,k}}+\nu_I|_{V_I^{c,k}})\pitchfork0$. Therefore, we will construct $\nu$ by the same inductive procedure as the proof of Proposition \ref{section} with the additional condition that
\begin{enumerate}[$(f)$]
\item $(s_I|_{V_I^{c,k}}+\nu_I|_{V_I^{c,k}})\pitchfork0$\textit{ for all }$|I|\leq k$.
\end{enumerate}
Assume inductively that suitable $\nu_I:V_I^{|I|}\rightarrow E_I$ have been found for $|I|\leq k$ satisfying conditions $(a)$-$(f)$. We need to construct $\nu_J$ for some $J$ with $|J|=k+1$. The construction of $\nu_J$ is done similar the proof of Proposition \ref{section} by constructing $\nu_J$ as a sum $\widetilde{\nu}_J+\widetilde{\nu}^c_\pitchfork + \nu_\pitchfork$, where now $\widetilde{\nu}^c_\pitchfork$ achieves transversality on $V_J^c$ and $\nu_\pitchfork$ achieves transversality on $V_J$.
The $\widetilde{\nu}_J$ is constructed exactly as in the proof of Proposition \ref{section}. However, the inductive hypothesis allows us to make stronger statements.
\begin{enumerate}[(A$'$)]
\item Remark (A) continues to hold and $s_J^c$ has the same components as $s_J$, so at each point of \mbox{$\overline{B_{\eta_k+\frac{1}{2}}^{c,J}(N_{JI}^{c,k+\frac{1}{2}})}\setminus N_{JI}^{c,k}$} there is at least one component of $(s_J^{c,j})_{j\in J\setminus I}$ that is nonzero.
\item The same strong admissability condition (B) for $\widetilde{\mu}_J^{c,j}$ holds.
\item In addition to $s_J+\widetilde{\nu}_J$ being transverse to $0$ on $B:=B_{\eta_k+\frac{1}{2}}^{J}(N_J^{k+\frac{3}{4}})$ as in remark (C), we also have that $s_J|_{V_J^{c,k}}+\widetilde{\nu}_J|_{V_J^{c,k}}$ is transverse to $0$ on $B^c:=B_{\eta_k+\frac{1}{2}}^{c,J}(N_J^{c,k+\frac{3}{4}})=B\cap V_J^{c,k}$. This holds because by (A$'$) and (B$'$), all zeros of $s_J|_{V_J^{c,k}}+\nu_J|_{V_J^{c,k}}$ must lie in the subset $N_J^{c,k+\frac{3}{4}}$ of $B^c$. But here, $\nu_J|_{V_J^{c,k}}$ is obtained by pulling back $\nu_J|_{V_J^{c,k}}, I\subsetneq J$, so the zeros are transverse by using condition $(e)$ of the inductive hypothesis. The set $B^c$ compactly contains the neighborhood $B'_c:=B_{\eta_k+1}^{c,J}(N_J^{c,k+1})$ of the core $N_J^{c,k+1}$ where compatibility conditions are required.
\item By the same reasoning as (D), $\|\widetilde{\nu}_J^c\|<\sigma$ and satisfies admissability condition (B), so
    \[V_J^{c,k+1}\cap(s_J+\widetilde{\nu}_J)^{-1}(0)\subset N_J^{c,k+\frac{1}{4}}\cup(\widetilde{C}_J^c\setminus\overline{W_J^c}).\]
\end{enumerate}
This defines $\widetilde{\nu}_J$ which satisfies conditions $(a), (c)$, and $(d)$. We need to achieve transversality conditions $(b)$ and $(f)$ by choosing $\nu_\pitchfork$ while keeping $\nu_\pitchfork$ small enough that $(d)$ remains true. Choose a relatively open neighborhood $W$ of $\overline{V_J^{k+1}}\setminus B$ in $\overline{V_J^{k+1}}$ such that
\[(s_J+\widetilde{\nu}_J)^{-1}(0)\cap W\subset \overline{V_J^{k+1}}\cap \widetilde{C}_J,\quad \overline{B'}\cap W=\emptyset,\]
\[(s_J^c+\widetilde{\nu}_J|_{V^{c,k+1}_J})^{-1}(0)\cap W\cap V^{c,k+1}_J\subset \overline{V_J^{c,k+1}}\cap\widetilde{C}_J^c,\quad \overline{B'_c}\cap W=\emptyset.\]
This can be done using the zero set conditions (D) and (D$'$) and the facts that $B'\sqsubset B\subset W_J, B'_c\sqsubset B^c\subset W_J^c$. By (C), transversality condition $(b)$ holds in $B$ and hence outside the compact subset $\overline{V^{k+1}_J}\setminus B\subset W$. By (C$'$), transversality condition $(e)$ holds in $B^c=B\cap U_J^c$ and hence outside the compact set $\overline{V^{c,k+1}_J}\setminus B^c$. Hence there is an open precompact subset $P\sqsubset W$ such that condition $(b)$ holds on $W\setminus P$ and condition $(f)$ holds on $(W\setminus P)\cap V^{c,k+1}_J$. We can apply the Transversality Extension Theorem of \cite{gp} to choose $\nu_\pitchfork^c$ a very small function on $V_J^{c,k+1}$ with support in $P\cap V_J^{c,k+1}$ and values in $E_J$ such that $s_J|_{V_J^{c,k+1}}+\widetilde{\nu}_J|_{V_J^{c,k+1}} + \nu_\pitchfork^c$ is transverse to $0$ on $W\cap V_J^{c,k+1}$ and hence all of $V_J^{c,k+1}$. Thus, condition $(f)$ is satisfied. Next, use a cutoff function to extend $\nu_\pitchfork^c$ to a small smooth function $\widetilde{\nu}^c_\pitchfork$ on all of $V_J^{k+1}$. The transversality of $s_J|_{V_J^{c,k+1}}+\widetilde{\nu}_J|_{V_J^{c,k+1}} + \nu_\pitchfork^c$ on $V_J^{c,k+1}$ clearly also implies the transversality of $s_J+\widetilde{\nu}_J+\widetilde{\nu}^c_\pitchfork$ on $V_J^{c,k+1}$. Transversality is an open condition, so $s_J+\widetilde{\nu}_J+\widetilde{\nu}^c_\pitchfork$ is also transverse on some $\varepsilon$ neighborhood $B_\varepsilon(V_J^{c,k+1})$ of $V_J^{c,k+1}$ in $V_J^{k+1}$. Finally, we can apply the Transversality Extension Theorem to choose $\nu_\pitchfork$ a very small smooth function with support in $P\setminus B_{\frac{\varepsilon}{2}}(V_J^{c,k+1})$ and values in $E_J$ such that $s_J+\widetilde{\nu}_J+\widetilde{\nu}^c_\pitchfork + \nu_\pitchfork$ is transverse on $W$ and hence all of $V_J^{k+1}$. Condition $(f)$ continues to be satisfied and $\widetilde{\nu}_\pitchfork$ and $\nu_\pitchfork$ can be chosen arbitrarily small so condition $(d)$ remains satisfied. This completes the construction of $\nu$.
\end{proof}

One can combine the methods of this and the previous section to produce the following lemma.
\begin{lemma}\label{c1transverse}
Let $\mathcal{K}_c$ be a tame $C^1$ SS transverse subatlas of a tame $C^1$ SS atlas $\mathcal{K}$.  Let $\mathcal{C}\sqsubset\mathcal{V}$ be nested reductions for $\mathcal{K}$ and $\mathcal{C}_c\sqsubset\mathcal{V}_c$ be nested reductions for $\mathcal{K}_c$ such that $\mathcal{C},\mathcal{C}_c$ and $\mathcal{V},\mathcal{V}_c$ are compatible transverse reductions. Let $0<\delta<\delta_{\mathcal{V}_c}$ and $0<\sigma<\sigma(\delta,\mathcal{V},\mathcal{C})$. Then there exists a $(\mathcal{V},\mathcal{C},\delta,\sigma)$-adapted $C^1$ SS perturbation $\nu$ of $s_{\mathcal{K}}|_{\mathcal{V}}$ such that $\nu_c := \nu|_{\mathcal{V}_c}$ is a $(\mathcal{V}_c,\mathcal{C}_c,\delta,\sigma)$- $C^1$ SS adapted perturbation of $s^c_{\mathcal{K}_c}|_{\mathcal{V}_c}$. Moreover, $\nu$ and $\nu_c$ satisfy the additional transversality property.
\begin{itemize}
\item $s_I|_{\mathcal{V}_I}+\nu_I:V_I\rightarrow E_I$ is SS transverse to $0$.
\item $s^c_I|_{\mathcal{V}_I}+\nu^c_I:V_I^c\rightarrow E_I$ is SS transverse to $0$.
\end{itemize}
\end{lemma}
\begin{proof}
The proof of this lemma combines the proofs of Lemma \ref{tsection} and Proposition \ref{c1section} above.
\end{proof}

\subsection{Sections in the complement of a divisor}\label{complementsections}
The following Proposition is extremely useful and roughly states that the perturbed zero set of a $0$-dimensional atlas can be chosen to avoid subsets of codimension at least $2$. For the definitions regarding $C^1$ SS atlases, see Section \ref{c1ss}.
\begin{proposition}\label{codim2}
Let $\mathcal{K}$ be an oriented, $0$-dimensional, weak, effective $C^1$ SS Kuranishi atlas on a compact metrizable stratified space $(X,\mathcal{T})$. Then $\mathcal{K}$ determines a virtual fundamental class $[X]_{\mathcal{K}}^{vir}\in \check{H}_0(X;\mathbb{Q})$ and one can choose a perturbation section so that the perturbed zero set $|\mathbf{Z}_\nu|_{\mathcal{H}}$ lies the realization of the top stratum $|\mathcal{K}^{top}|$.
\end{proposition}
\begin{proof}
The proof of Proposition \ref{codim2} is exactly as the proof of \cite[Proposition 3.5.7]{notes}.
\end{proof}

\bibliographystyle{amsalpha}
\bibliography{research}

\end{document}